\documentclass{amsart}

\usepackage{verbatim}
\usepackage{xcolor}
\usepackage{tikz}
\usetikzlibrary{cd}
\usepackage{arydshln}
\usepackage{caption}
\usepackage{subcaption}
\usepackage{graphicx}
\usepackage{gensymb}
\usepackage{mathrsfs}
\usepackage{appendix}
\usepackage{bbm}

\let\stdthebibliography\thebibliography
\let\stdendthebibliography\endthebibliography
\renewenvironment*{thebibliography}[1]{%
  \stdthebibliography{ABC+20}}
  {\stdendthebibliography}

\tikzset{
	symbol/.style={
		draw=none,
		every to/.append style={
			edge node={node [sloped, allow upside down, auto=false]{$#1$}}}
	}
}

\usepackage{amsmath}
\usepackage{bbm}
\usepackage{hyperref}

\newtheorem{theorem}{Theorem}[section]
\newtheorem{proposition}[theorem]{Proposition}
\newtheorem{lemma}[theorem]{Lemma}
\newtheorem{corollary}[theorem]{Corollary}

\theoremstyle{definition}
\newtheorem{definition}[theorem]{Definition}

\theoremstyle{remark}
\newtheorem{remark}[theorem]{Remark}

\numberwithin{equation}{section}

\newcommand{\agamma}{\vec{\gamma}}

\newcommand{\RR}{\mathbb{R}}

\DeclareMathOperator{\Mod}{Mod}

\newcommand{\T}{\mathcal{T}}
\newcommand{\M}{\mathcal{M}}
\DeclareMathOperator{\Stab}{Stab}

\newcommand{\cQ}{\mathcal{Q}}

\newcommand{\QoT}{\mathcal{Q}^1\mathcal{T}_g}
\newcommand{\QT}{\mathcal{Q}\mathcal{T}_g}

\newcommand{\QoM}{\mathcal{Q}^1\mathcal{M}_g}
\newcommand{\MF}{\mathcal{MF}}

\newcommand{\MRG}{\mathcal{MRG}}
\newcommand{\Fol}{\mathcal{F}}
\newcommand{\bfL}{\mathbf{L}}

\newcommand{\vg}{\vec{\gamma}}
\newcommand{\mrg}{\mathbf{x}}
\newcommand{\mrgy}{\mathbf{y}}

\renewcommand{\tt}{\mathcal{T}_g}
\newcommand{\mcg}{\mathrm{Mod}_g}

\DeclareMathOperator{\JS}{JS}
\DeclareMathOperator{\Ext}{Ext}
\DeclareMathOperator{\Teich}{Teich}
\newcommand{\AGY}{\mathrm{AGY}}
\newcommand{\MV}{\mathrm{\mathrm{MV}}}
\newcommand{\Th}{\mathrm{Thu}}
\newcommand{\Sp}{\Xi}
\newcommand{\bfm}{\mathbf{m}}
\newcommand{\bfx}{\mathbf{x}}
\newcommand{\tbfm}{\widetilde{\mathbf{m}}}
\newcommand{\hbfm}{\widehat{\mathbf{m}}}
\newcommand{\tnu}{\widetilde{\nu}}
\newcommand{\hnu}{\widehat{\nu}}
\newcommand{\tmu}{\widetilde{\mu}}

\begin{document}
\title[Critical graphs of Jenkins-Strebel differentials]{The distribution of critical graphs of Jenkins-Strebel differentials}

\author{Francisco Arana--Herrera}

\email{farana@umd.edu}

\address{Department of Mathematics, University of Maryland, 4176 Campus Drive, College Park, MD 20742, USA}
	
\author{Aaron Calderon}

\email{aaroncalderon@uchicago.edu}

\address{Department of Mathematics, University of Chicago, 5734 S. University Ave, Chicago, IL 60637, USA}


\date{\today}

\begin{abstract}
By work of Jenkins and Strebel, given a Riemann surface $X$ and a simple closed multi-curve $\alpha$ on it, there exists a unique quadratic differential $q$ on $X$ whose horizontal foliation is measure equivalent to $\alpha$. We study the distribution of the critical graphs of these differentials in the moduli space of metric ribbon graphs as the extremal length of the multi-curves goes to infinity, showing they equidistribute to the Kontsevich measure regardless of the initial choice of $X$.
\end{abstract}

\maketitle

\thispagestyle{empty}

\section{Introduction}
\label{sec:intro}

\subsection*{Overview.} Famous independent works of Jenkins and Strebel \cite{Jenkins,Streb1,Streb2,Streb3} show that, given a Riemann surface $X$ and a simple closed multi-curve $\alpha$ on it, there exists a unique quadratic differential $q$ on $X$ whose horizontal foliation represents each curve of $\alpha$ by a cylinder of height one. To every such quadratic differential $q$ one can associate its {\em critical graph}, a ribbon graph having the singularities of $q$ as vertices and the horizontal saddle connections of $q$ as edges. This ribbon graph inherits a metric from the singular flat metric induced by $q$ on $X$. Roughly speaking, the critical graph encodes the conformal geometry of $X$ away from $\alpha$. See Figure \ref{fig:JS}.

By work of Mirzakhani\footnote{Although not explicitly stated, this follows directly from \cite[Theorem 1.3]{Mir08b}. See
\cite[Corollary 5.13]{MGTcurrents} for another proof from the viewpoint of geodesic currents.}
\cite{Mir08b},
given a closed, connected, oriented Riemann surface $X$, the number of simple closed (multi-)curves on $X$ of extremal length $\leq L^2$ is asymptotic to a polynomial of degree $6g-6$.
Each one of these curves gives rise to a critical graph as described above.
In this paper we show these critical graphs, appropriately rescaled, equidistribute to the Kontsevich measure on the moduli space of metric ribbon graphs. In particular, the limiting distribution is independent of the initial choice of $X$.

\begin{figure}[hb]
    \centering
    \includegraphics[scale=.5]{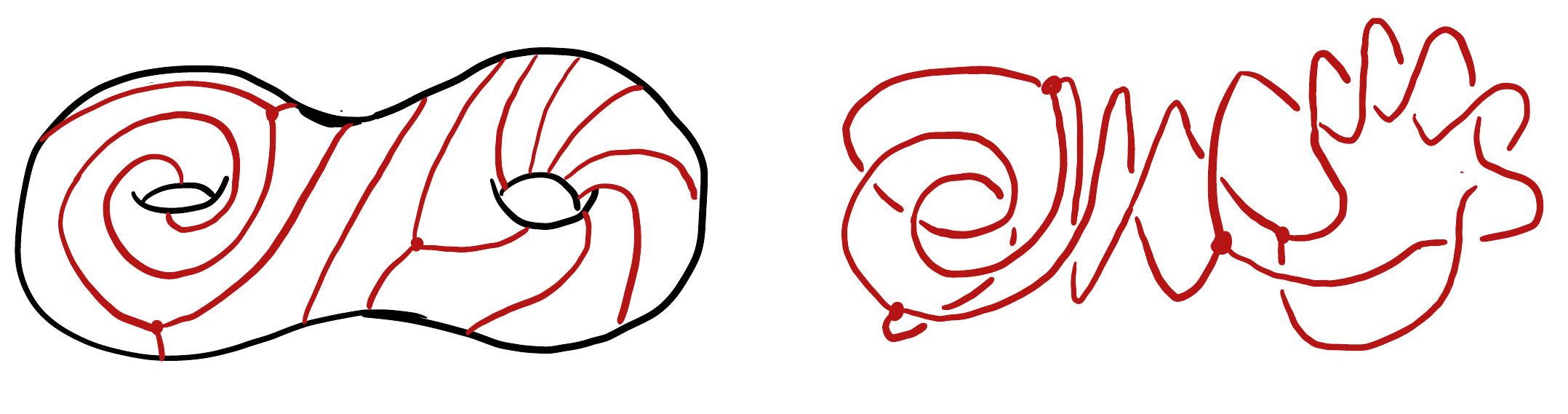}
    \caption{A Jenkins-Strebel differential and its critical graph.}
    \label{fig:JS}
\end{figure}

The question at hand is heavily motivated by analogous results in the setting of hyperbolic surfaces. In \cite{spine}, the authors showed that, given a closed, connected, oriented hyperbolic surface $X$, the metric ribbon graph spines of complementary subsurfaces to simple closed multi-geodesics also equidistribute to the Kontsevich measure on the corresponding moduli space as the lengths of the geodesics goes to infinity.
This result in turn follows a line of investigation that can be traced back to several other authors \cite{Mir16, Liu19,Ara20a,ES20}.
Analogous questions have also been answered in the setting of homogeneous dynamics
\cite{AES16a,AES16b,ERW19}.

Paralleling our approach in \cite{spine}, we
reduce the problem to an equidistribution question concerning the dynamics of the Teichmüller horocycle flow on moduli spaces of quadratic differentials.
These reductions combine Margulis's well-known averaging and unfolding techniques \cite{Mar04} with several recent developments in Teichmüller theory. Of particular importance are
Delaunay triangulations of quadratic differentials following Masur and Smillie \cite{MS91}, the AGY-metric on the moduli space of quadratic differentials developed by Avila, Gouzel, and Yoccoz \cite{AGY06,AG13}, and the study of the projection of the Masur-Veech measure to the moduli space of Riemann surfaces carried out by Athreya, Bufetov, Eskin, and Mirzakhani \cite{ABEM12}.

The main result of this paper also provides a new procedure for sampling random metric ribbon graphs. In particular, the geometry of any single conformal surface reflects the geometry of random metric ribbon graphs.

\subsection*{Main result.} To streamline our exposition, for the moment we only state our main result in the (representative) case of non-separating simple closed curves.
The statement for the general case appears as Theorem \ref{theo:main_2A} below.

For the rest of this discussion fix an integer $g \geq 2$ and denote by $S_g$ a closed, connected, oriented surface of genus $g$. Let $\mathrm{Mod}_g$ be the mapping class group of $S_g$, $\mathcal{T}_g$ be the Teichmüller space of marked conformal structures on $S_g$, and $\mathcal{M}_g$ be the moduli space of conformal structures on $S_g$. Free homotopy classes of unoriented simple closed curves on $S_g$ will be refered to as simple closed curves. Given a simple closed curve $\alpha$ on $S_g$ and a marked conformal structure $X \in \mathcal{T}_g$, denote by $\mathrm{Ext}_X(\alpha) > 0$ the extremal length of $\alpha$ with respect to $X$; see \S\ref{sec:prelim} below for a discussion of the different definitions of extremal length.

Let $\gamma$ be a non-separating simple closed curve on $S_g$ and $X \in \mathcal{T}_g$ be a marked conformal structure on $S_g$. For every $L > 0$ consider the counting function
\[
s(X,\gamma,L) := \# \{\alpha \in \mathrm{Mod}_g \cdot \gamma \ | \ \mathrm{Ext}_X(\alpha) \leq L^2 \}.
\]
This function does not depend on the marking of $X \in \mathcal{T}_g$ but only on its underlying conformal strucure $X \in \mathcal{M}_g$. Indeed, it is equal to the number of non-separating simple closed curves on $X$ of extremal length $\leq L^2$.
By Mirzakhani's seminal work \cite{Mir08b},
this counting function is asymptotically polynomial (see \eqref{eq:mirA} below).

Given $X \in \mathcal{T}_g$ and a non-separating simple closed curve $\alpha$ on $S_g$,
denote by $\JS(X, \alpha)$ the unique Jenkins--Strebel differential on $X$ whose horizontal measured foliation is equivalent to $\alpha$ (where each component is given weight 1); in particular, the horizontal foliation of $\JS(X, \alpha)$ consists of a single cylinder of height 1.
Consider the {\em critical graph} of this differential, that is, the union of all of the singular horizontal trajectories of $\JS(X, \alpha)$ equipped with the restriction of the underlying singular flat metric and the ribbon structure coming from its embedding in $S_g$.
Each of its boundaries has length $\Ext_X(\alpha)$, so to understand the distribution of these graphs we need to rescale them (see \S \ref{sec:prelim}).

Denote by $\mathcal{MRG}_{g-1,2}(1,1)$ the moduli space of metric ribbon graphs of genus $g-1$ with two boundary components, each of length $1$. For $X$ and $\alpha$ as above let 
\[\Sp^1(X, \alpha) \in \mathcal{MRG}_{g-1,2}(1,1)\]
denote the critical graph of $\JS(X, \alpha)$ rescaled by $1/\Ext_X(\alpha)$.
We remark that $\Sp^1(X, \alpha)$ is {\em not} the same as the critical graph of the unit-area Jenkins--Strebel differential whose horizontal foliation is projectively equivalent to $\alpha$; this is because critical graphs scale linearly while extremal length scales quadratically.

On $\mathcal{MRG}_{g-1,2}(1,1)$ consider the counting measure
\[
\eta_{X,\gamma}^L := \sum_{\alpha \in \mathrm{Mod}_g \cdot \gamma} \mathbbm{1}_{[0,L^2]}(\mathrm{Ext}_X(\alpha)) \cdot \delta_{\Sp^1(X, \alpha)}.
\]
Just like the counting function $s(X,\gamma,L)$, this measure does not depend on the marking of $X \in \mathcal{T}_g$ but only on the underlying conformal structure. Denote by $\eta_\mathrm{Kon}$ the measure induced by the Kontsevich symplectic form on $\mathcal{MRG}_{g-1,2}(1,1)$ and by $c_g > 0$ its total mass (see \S \ref{sec:critical_equi} below or \cite[\S2]{spine}).

The following theorem, which shows that the rescaled critical graphs of Jenkins--Strebel differentials of  non-separating simple closed curves equidistribute over the moduli space $\mathcal{MRG}_{g-1,2}(1,1)$, is an instance of the main result of this paper. For the general version see Theorem \ref{theo:main_2A}, as well as Theorem \ref{theo:main_3_new} for an even stronger version concerning simultaneous equidistribution.

\begin{theorem}
\label{theo:main_intro}
Let $\gamma$ be a non-separating simple closed curve and $X \in \mathcal{M}_g$. Then
\[
\lim_{L \to \infty} \frac{\eta_{X,\gamma}^L}{s(X,\gamma,L)} = \frac{\eta_{\mathrm{Kon}}}{c_g}
\]
with respect to the weak-$\star$ topology for measures on $\mathcal{MRG}_{g-1,2}(1,1)$,
\end{theorem}

\subsection*{Main ideas of proof.}
Our proof of Theorem \ref{theo:main_intro} follows the same outline as the analogous result in the hyperbolic setting \cite[Theorem 1.1]{spine}.
Namely, we reduce the desired equidistribution result on $\MRG_{g-1,2}(1,1)$ to a curve counting result, which we then translate back into an equidistribution result over $\M_g$. 

First, consider the following reformulation.
Let $f \colon \mathcal{MRG}_{g-1,2}(1,1) \to \mathbb{R}_{\geq 0}$ be a continuous, compactly supported function.
Then, for any non-separating simple closed curve $\gamma$, any $X \in \mathcal{T}_g$, and any $L > 0$, define the $f$-weighted counting 
\begin{align*}
c(X,\gamma,f,L) &:= 
\sum_{\alpha \in \mathrm{Mod}_g \cdot \gamma} \mathbbm{1}_{[0,L^2]}(\Ext_X(\alpha)) \cdot f(\Sp^1(X, \alpha))\\
&\phantom{:}=
\int_{\mathcal{MRG}_{g-1,2}(1,1)} f(\bfx) \thinspace d\eta_{X,\gamma}^L(\bfx).
\end{align*}

Theorem \ref{theo:main_intro} is then equivalent to the following counting result.

\begin{theorem}
\label{theo:red_introA}
Let $\gamma$ be a non-separating simple closed curve on $S_g$, $X \in \mathcal{M}_g$ be a conformal structure on $S_g$, and $f \colon \mathcal{MRG}_{g-1,2}(1,1) \to \mathbb{R}_{\geq 0}$ be a continuous, compactly supported function. Then,
\[
\lim_{L \to \infty} \frac{c(X,\gamma,f,L)}{s(X,\gamma,L)} = \frac{1}{c_g} \int_{\mathcal{MRG}_{g-1,2}(1,1)} f(\bfx) \thinspace d\eta_\mathrm{Kon}(\bfx).
\]
\end{theorem}

Once in this setting, we apply Margulis's ``averaging and unfolding'' techniques \cite{Mar04} to reduce the counting problem at hand to an equidistribution question over $\M_g$.
Just as in \cite{spine}, during the averaging step one needs  uniform control over how the metric ribbon graph $\Sp^1(X, \alpha)$ varies as $X \in \mathcal{T}_g$ does.
The basic issue is that if $\JS(X, \alpha)$ lies in (or near) a non-principal stratum then small deformations of $X$ can change the topology of the critical graph.
To remedy this, we prove in Proposition \ref{prop:key_estimate} that for most $\alpha$, the differential $\JS(X, \alpha)$ lies far away from non-principal strata.
Proposition \ref{prop:varA} then shows that for such $\alpha$ we can achieve the desired uniform control on the geometry of the critical graph.

This control allows us to perform the averaging and unfolding step of our argument, after which our original problem reduces to a question regarding the equidistribution of certain subsets, which we call ``critical-JS-horoballs,'' in the moduli space $\mathcal{M}_g$.
To tackle this question we use the rich dynamics of the Teichmüller geodesic and horocycle flows. More concretely, we use work of Forni \cite{push} which in turn relies crucially on work of Eskin, Mirzakhani, and Mohammadi \cite{EM,EMM}.

\subsection*{Related results}
As mentioned above, the main result of this paper directly parallels the main theorem of \cite{spine}. Indeed, as a consequence of our results, we find that the limiting distribution of complementary subsurfaces to hyperbolic geodesics and critical graphs of Jenkins--Strebel differentials are exactly the same.
This is a reflection of the phenomenon that as the boundary lengths of hyperbolic surfaces go to infinity, they look more and more like ribbon graphs (compare to \cite{Do10} and \cite{Mo09}).

There is a rich history of using both the conformal and hyperbolic incarnations of moduli space to probe its structure, often resulting in analogous theorems.
Of particular relevance to this paper are two different identifications of the moduli spaces of punctured/bordered Riemann surfaces with the space of metric ribbon graphs.
The first, due to Harer, Mumford, Penner, and Thurston, uses Jenkins--Strebel differentials with specified poles and residues; this identification was used in the computation of the orbifold Euler characteristic of $\M_{g,n}$ \cite{HZeuler} as well as Kontsevich's proof of Witten's conjecture \cite{Kon92}. 
The second, due to Do and Luo, uses the spine of a hyperbolic surface with boundary (see \cite{Luo}, \cite{Do10}, as well as \cite{Mo09}). This can be used to unite Mirzkahani's proof of Witten's conjecture \cite{Mir07c} with Kontsevich's (see \cite{Do10}). 
Our previous work \cite{spine} dealt with this second identification, while the paper at hand corresponds to the first paradigm.

Our theorem also echoes other equidistribution results on moduli space. 
The Teichm{\"u}ller geodesic flow, represented by the action of the diagonal matrix
\[
g_t := \left(\begin{array}{c c}e^t & 0 \\ 0 & e^{-t}\end{array} \right),
\]
expands the length of horizontal separatrices by a factor of $e^t$, so for every simple closed curve $\alpha$ the rescaled critical graph $\Xi^1(X, \alpha)$ is the same as the critical graph of $\smash{g_{-\log \Ext_X(\alpha)} \JS(X, \alpha)}$.
Denote by $\|q\|$ the area of a quadratic differential $q$. With this identification, Theorem \ref{theo:main_intro} states that point masses on the critical graphs of
\begin{equation}\label{eqn:JSflowback}
\{
g_{-\log \|q\|} q
\, | \,
q = \JS(X, \alpha) \text{ for } \alpha \in \Mod_g \cdot \gamma 
\text{ and }
\|q\| \le L
\}
\end{equation} 
equidistribute in $\MRG_{g-1, 2}(1,1)$ as $L$ tends to infinity.

This should be compared with work of \cite{ABEM12} and the first author \cite{ABEMeff}, which give mean equidistribution theorems for expanding Teichm{\"u}ller balls in moduli space.
More precisely, for every $X \in \T_g$, one can consider the underlying Riemann surfaces of the expanding ball
\begin{equation}\label{eqn:Teichball}
\{g_{\log \|q\|} q
\, | \,
q \in Q(X)
\text{ and }
\|q\| \le L\},
\end{equation}
where $Q(X)$ is the vector space of holomorphic quadratic differentials on $X$;
this is the same as the set of points with Teichm{\"u}ller distance at most $\log L$ from $X$.
Theorem 1.2 of \cite{ABEM12} then states that on average, the image of these balls in $\M_g$ equidistribute as $L$ tends to infinity.
The link between these equidistribution results for \eqref{eqn:JSflowback} and \eqref{eqn:Teichball} is the uniform distribution of simple closed curves on the space $\mathcal{MF}_g$ of singular measured foliations on $S_g$ \cite[Theorem 1.3]{Mir08b}.

Lastly, it is interesting to contrast our theorem with the main result of \cite{DSdense}. There, Dozier and Sapir show that the projections of some strata of $\QoM$ to $\M_g$ are not coarsely dense; in particular, the geometry of the critical graphs of horizontally periodic unit area differentials may be strongly constrained.
For example, if $X \in \M_g$ is not near the projection of the minimal stratum then it supports no horizontally periodic unit area differential whose critical graph has a vertex of valence $4g-2$, and supports no differential whose critical graph is close to such.

On the other hand, our main theorem implies that the set of critical graphs for a fixed $X$, rescaled to have unit length, is dense in the space of metric ribbon graphs. 
Since critical graphs are scaled under the Teichm{\"u}ller geodesic flow, our result can also be interpreted as saying that any $X$ supports a Jenkins--Strebel differential $q=\JS(X, \alpha)$ so that $g_{-\log \|q\|} q$ is arbitrarily close to the minimal stratum.

\subsection*{Organization of paper.}
In \S\ref{sec:prelim} we collect preliminary material needed to understand the statements and proofs of the main results of the paper. In addition to standard background material, we focus on how taking the critical graph of a Jenkins--Strebel differential gives a map to the moduli space of metric ribbon graphs. 
In \S\ref{sec:msl} we show that the Jenkins--Strebel differentials of most simple closed multi-curves stay deep within the principal stratum; this allows us in the subsequent \S\ref{sec:weightsvaryA} to invoke results about the AGY metric to show that for most curves, the critical graphs of their Jenkins--Strebel differentials vary uniformly as the base surface varies.
In \S\ref{sec:critical_equi} we define ``critical-JS-horoballs'', compute their total mass, and show that they equidistribute over moduli space. The results in this section rely on transporting measures between $\mathcal{T}_g$ and spaces of quadratic differentials and leverage the ergodic theory of the $\mathrm{SL}(2,\mathbb{R})$ action on the latter spaces.
Finally, in \S\ref{sec:avgunfold} we apply Margulis's averaging and unfolding strategy to prove Theorem \ref{theo:main_2A}, a generalization of Theorem \ref{theo:main_intro} to arbitrary multi-curves.

\subsection*{Acknowledgements.} The authors would like to thank Vincent Delecroix for suggesting the question answered in this paper and Curt McMullen for helpful comments. The authors would also like to thank Giovanni Forni for very enlightening conversations.
AC acknowledges support from NSF grant DMS-2202703.

\vfill
\pagebreak

\tableofcontents

\section{Preliminaries}\label{sec:prelim}

\subsection*{Outline of this section.} In this section we give of a brief overview of the basic objects and theorems that will be used throughout the rest of this paper. We begin with a discussion on the theory of Jenkins-Strebel differentials with a special emphasis on the work of Hubbard and Masur \cite{HM79}. We then briefly recall the moduli spaces of metric ribbon graphs; for a more detailed discussion of these spaces see \cite{spine}.
Lastly, we discuss in detail how the construction of critical graphs of Jenkins-Stebel differentials defines a map from (quotients of) Teichmüller space to  appropriate moduli spaces of metric ribbon graphs.

\subsection*{Extremal length.} Given a Riemann surface $X$ and a simple closed curve $\gamma$ on it, the extremal length of $\gamma$ with respect to $X$ admits two equivalent definitions. First, it can be defined analytically as
\begin{equation}
\label{eq:ext1}
\mathrm{Ext}_X(\gamma) := \sup_\rho \frac{\ell_\rho(\gamma)^2}{\mathrm{Area}(\rho)},
\end{equation}
where the supremum ranges over all conformal metrics $\rho$ on $X$ of non-zero, finite area and $\ell_\rho(\gamma)$ denotes the infimum of the $\rho$-lengths of simple closed curves isotopic to $\gamma$. Equivalently, it can be defined geometrically as
\begin{equation}
\label{eq:ext2}
\mathrm{Ext}_X(\gamma) := \inf_{C} \frac{1}{\mathrm{mod}(C)},
\end{equation}
where the infimum ranges over all embedded cylinders $C$ on $X$ with core curve isotopic to $\gamma$ and  $\mathrm{mod}(C)$ denotes the modulus of the cylinder $C$.

In independent work, Jenkins and Strebel \cite{Jenkins,Streb1,Streb2,Streb3} showed these two {\em a priori} different notions of extremal length are equivalent through the construction of so-called Jenkins-Strebel differentials; see Theorem \ref{theo:js} below.


\subsection*{Quadratic differentials.} A (holomorphic) quadratic differential $q$ on a Riemann surface $X$ is a differential which in local coordinates has the form $f(z) \thinspace dz^2$ for some holomorphic function $f(z)$. Such a differential has a well defined notion of area,
\[
\|q\| := \mathrm{Area}(q) := \int_X |q|.
\]
More precisely, the differential $q$ induces a singular flat metric on $X$. If in local coordinates $z =  x + iy$ then the metric is given by $dx^2 + dy^2$; the zeroes of the differential correspond to singularities of the metric. The area of $q$ is the total area of this metric. We denote by $Q(X)$ the complex vector space of all quadratic differentials on a Riemann surface $X$, and by $S(X) \subseteq Q(X)$ the sphere of all unit area quadratic differentials on $X$. We sometimes denote quadratic differentials by $(X,q)$ to record the Riemann surface $X$ on which they are defined.

The spaces $Q(X)$ and $S(X)$ for $X$ ranging over the Teichm{\"u}ller space $\mathcal{T}_g$ can be arranged into bundles $\mathcal{QT}_g$ and $\mathcal{Q}^1\mathcal{T}_g$ over $\mathcal{T}_g$ of marked (unit area) quadratic differentials; note that ``marking'' here refers to a marking only of the underlying surface $S_g$ and not the zeros of the quadratic differential.
These bundles support natural $\mathrm{Mod}_g$ actions given by change of markings. 
The quotient $\mathcal{Q}^1\mathcal{M}_g := \mathcal{Q}^1\mathcal{T}_g/\mathrm{Mod}_g$ is the bundle of unit area quadratic differentials over moduli space.

The bundle $\QT$ carries a natural Lebesgue-class measure called the {\em Masur--Veech measure} which is induced from an integral lattice, corresponding to differentials tiled by unit area squares. This measure induces a measure $\nu_{\mathrm{MV}}$ on the hypersurface $\QoT$ also called the Masur--Veech measure. Both versions of this measure are $\Mod_g$ invariant, and the corresponding pushforward $\widehat{\nu}_{\mathrm{MV}}$ on $\QoM$ is a finite Lebesgue-class measure. Throughout the paper we will denote its mass by $b_g$.

\subsection*{Singular measured foliations.}
Denote by $\mathcal{MF}_g$ the space of singular measured foliations on $S_g$ up to isotopy and Whitehead moves. The set of isotopy classes of weighted simple closed curves on $S_g$ embeds densely into $\mathcal{MF}_g$. Furthermore, geometric intersection number extends continuously to a pairing on $\mathcal{MF}_g$. Train track coordinates (see \S \ref{sec:msl}) induce a natural integral piecewise linear structure on $\mathcal{MF}_g$. In particular, $\mathcal{MF}_g$ carries a natural Lebesgue class measure $\mu_\mathrm{Th}$ called the Thurston measure which is invariant under the natural $\mathrm{Mod}_g$-action on $\mathcal{MF}_g$. Denote by $\mathcal{PMF}_g$ the projectivization of $\mathcal{MF}_g$ under the action that scales transverse measures and by $[\lambda] \in \mathcal{PMF}_g$ the projective class of $\lambda \in \mathcal{MF}_g$.

Every quadratic differential $q$ on a Riemann surface $X$ gives rise to a pair of singular measures foliations $\Re(q)$ and $\Im(q)$ on $X$. If in local coordinates $z =  x + iy$ the differential $q$ corresponds to $dz^2$ then $\Re(q)$ corresponds to the measured foliation induced by $|dx|$ while $\Im(q)$ corresponds to the measured foliation induced by $|dy|$; the zeroes of $q$ correspond to the singularities of the foliations. We refer to $\Re(q)$ and $\Im(q)$ as the vertical and horizontal foliations of $q$. These constructions give rise to $\mathrm{Mod}_g$-equivariant maps $\Re, \Im \colon \mathcal{QT}_g \to \mathcal{MF}_g$.

\subsection*{Jenkins-Strebel differentials.} It is natural to ask whether, given a marked Riemann surface $X \in \mathcal{T}_g$ and a simple closed curve $\gamma$ on $S_g$, it is possible to find a marked quadratic differential $q \in Q(X)$ such that $\Im(q) = \gamma$. Jenkins and Strebel \cite{Jenkins,Streb1,Streb2,Streb3} independently showed that this is always possible and, moreover, in a unique way. We refer to the corresponding quadratic differential $\mathrm{JS}(X,\gamma) \in Q(X)$ as the {\em Jenkins-Strebel differential} of $\gamma$ on $X$. Jenkins and Strebel's motivation can be further understood through the following result.

\begin{theorem}
   \label{theo:js}
   Given a marked complex structure $X \in \mathcal{T}_g$ and a simple closed curve $\gamma$ on $S_g$, the singular flat metric induced by $\mathrm{JS}(X,\gamma) \in Q(X)$ realizes the supremum in \eqref{eq:ext1}. Furthermore, the complement of the critical leaves of vertical foliation of $\mathrm{JS}(X,\gamma) \in Q(X)$ is a cylinder realizing the infimum in \eqref{eq:ext2}. In particular, the supremum and infimum in \eqref{eq:ext1} and \eqref{eq:ext2} are equal.
\end{theorem}

\subsection*{The Hubbard-Masur theorem.} The question addressed by the works of Jenkins and Strebel can also be considered for more general singular measured foliations. Indeed, in \cite{HM79} Hubbard and Masur proved the following structural result.

\begin{theorem}
    \label{theo:HM1}
    Given  $X \in \mathcal{T}_g$ and a singular measured foliation $\lambda \in \mathcal{MF}_g$, there exists a unique quadratic differential $q = q(X,\lambda) \in Q(X)$ such that $\Im(q) = \lambda$. Furthermore, the map $q \in Q(X) \mapsto \Im(q) \in \mathcal{MF}_g$ is a homeomorphism.
\end{theorem}

Said another way, the map associating to a simple closed curve its Jenkins--Strebel differential extends to a homeomorphism $\JS: \T_g \times \MF_g \to \QT$.

Theorem \ref{theo:HM1} suggests the following extension of the notion of extremal length to singular measured foliations: the extremal length $\mathrm{Ext}_X(\lambda)$ of $\lambda \in \mathcal{MF}_g$ with respect to $X \in \mathcal{T}_g$ is the area of $\JS(X,\lambda) \in Q(X)$. With this definition, extremal length is $2$-homogeneous with respect to the scaling action on transverse measures; this prompts us to usually work with the square root of extremal length rather than with extremal length itself. For other methods of extending extremal lengths to singular measured foliations (and more general objects) see \cite{Ker80,MGTcurrents,emcd3}.

One can also ask which pairs of singular measured foliations can arise as the vertical and horizontal foliations of quadratic differentials. It turns out that as long as the pair of foliations appropriately fills the surface, this is always possible \cite{GM91}.
A pair of singular measured foliations $(\lambda,\mu) \in \mathcal{MF}_g \times \mathcal{MF}_g$ is said to {\em bind} the surface $S_g$ if their geometric intersection number with any singular measured foliation is positive. Denote by $\Delta \subseteq \mathcal{MF}_g \times \mathcal{MF}_g$ the set of pairs of non-binding singular measured foliations.

\begin{theorem}\label{theo:GM}
    Given a pair of binding singular measured foliations $\mu, \lambda \in \mathcal{MF}_g$, there exists a unique quadratic differential $q \in \mathcal{QT}_g$ such that $\Re(q) = \mu$ and $\Im(q) = \lambda$.
    Furthermore, the map $(\Re,\Im) \colon \mathcal{QT}_g \to \mathcal{MF}_g \times \mathcal{MF}_g \setminus \Delta$ is a homeomorphism.
\end{theorem}

\subsection*{The Teichmüller metric.} The Teichmüller metric $d_\mathrm{Teich}$ on $\mathcal{T}_g$ quantifies the minimal dilation among quasiconformal maps between marked complex structures on $S_g$. More precisely, for every $X,Y \in \mathcal{T}_g$,
\[
d_{\Teich}(X,Y) := \frac{1}{2} \log\left( \inf_{f \colon X \to Y} K(f)\right),
\]
where the infimum runs over all quasiconformal maps $f \colon X \to Y$ in the homotopy class given by the markings of $X$ and $Y$, and where $K(f)$ denotes the dilation of such maps. See \cite[Chapter 11]{FM11} for a more detailed definition. The action of the mapping class group on Teichmüller space is the isometry group of the Teichmüller metric \cite{Royden}. The Teichmüller metric is complete and its geodesics can be described explicitly in terms of the diagonal part of the natural $\mathrm{SL}(2,\mathbb{R})$-action on $\mathcal{Q}^1\mathcal{T}_g$.


In \cite{Ker80}, Kerckhoff proved the following formula for the Teichmüller metric; this formula provides control over the extremal lengths of singular measured foliations in terms of the Teichmüller distance between two marked complex structures.

\begin{theorem}\label{thm:Kerckhoff}
For any pair of marked complex structures $X, Y \in \T_g$, 
\begin{equation}
\label{eq:d}
d_\mathrm{Teich}(X,Y) = \frac{1}{2}
\max_{\lambda \in \MF_g} \log \left( \frac{\Ext_Y(\lambda)}{\Ext_X(\lambda)} \right).
\end{equation}
Furthermore, if $q = q(X,Y) \in S(X)$ is the unit area quadratic differential corresponding to the unique Teichmüller geodesic from $X$ to $Y$, then $\Im(q)$ realizes the maximum in \eqref{eq:d}.
\end{theorem}

The fact that $\Im(q)$ realizes the supremum (as opposed to $\Re(q)$) is due to the fact that Teichm{\"u}ller mappings stretch the leaves of $\Im(q)$ while shrinking its measure.
For example, suppose $q\in S(X)$; then $\lambda = \Im(q)$ has unit extremal length. Then its image $g_t q$ under the Teichm{\"u}ller geodesic flow still has area 1 and $\Im(g_t q) = e^{-t} \lambda$. Thus, $\lambda$ has extremal length $e^{2t}$ on $g_t q$.

\subsection*{Metric ribbon graphs.}
A ribbon graph is the combinatorial data of a deformation retraction of a surface with boundary.
More formally, it is a (simplicial) graph $\Gamma$ equipped with a cyclic ordering of the edges at each vertex; this can be reconciled with the first notion by thickening each edge to a ribbon and gluing the edges of the ribbons according to the cyclic ordering.
The genus and number of boundary components of $\Gamma$ are the values for the resulting topological surface. One may equip a ribbon graph $\Gamma$ with a metric $\mathbf{x}$ assigning a length to each of its edges to obtain a so-called metric ribbon graph $(\Gamma,\mathbf{x})$.

For $b \ge 0$ a non-negative integer, denote by $\MRG_{g,b}$ the moduli space of all metric ribbon graphs with genus $g$ and $b$ distinctly-labeled boundary components, all of whose vertices have valence at least three. 
For a given tuple $\bfL = (L_1, \ldots, L_b)$ of positive numbers, define $\MRG_{g,b}(\bfL) \subseteq \mathcal{MRG}_{g,b}$ to be the subset of all ribbon graphs whose (labeled) boundary components have lengths $L_1, \ldots, L_b$. The slices $\MRG_{g,b}(\bfL)$ piece together to form a fibration
$\MRG_{g,b} \to \smash{\RR_{>0}^{b}}$.

\subsection*{Critical graphs.} Given a marked complex structure $X \in \mathcal{T}_g$ and a simple closed curve $\gamma$ on $S_g$, the critical graph of the Jenkins-Strebel differential $q = \mathrm{JS}(X,\gamma) \in Q(X)$ is the metric ribbon graph obtained as the union of the critical leaves of $\Im(q) \in \mathcal{MF}_g$ endowed with the restriction of the singular flat metric induced by $q$ on $X$. A similar construction can be carried out for general multi-curves, but to get a well defined map to a corresponding moduli space of metric ribbon graphs, care needs to be exercised with regards to the symmetries of the construction.

For the rest of this discussion let $\vg := (\vg_1,\dots,\vg_k)$ be an ordered, oriented simple closed multi-curve on $S_g$, let $\mathrm{Stab}_0(\vg) \subseteq \mathrm{Mod}_g$ be its oriented stabilizer, i.e. the set of mapping classes that fix each component of $\vg$ together with its orientation, and let $\gamma \in \mathcal{MF}_g$ be its equivalence class as a singular measured foliation. Cutting $S_g$ along the components of $\vg$ yields a (possibly disconnected) surface with boundary $S_g \setminus \vg$. Label the components of $S_g \setminus \vg$ by $(\Sigma_j)_{j=1}^c$; such a labeling is possible because the components of $\gamma$ are labeled and oriented. For each $j \in \{1,\dots,c\}$ let $g_j,b_j \geq 0$ to be non-negative integers such that $\Sigma_j$ is homeomorphic to $\smash{S_{g_j,b_j}}$.

For any length vector $\bfL \in \mathbb{R}_{>0}^k$ we set
\[
\MRG(S_g \setminus \agamma; \bfL) :=
\prod_{j=1}^c \MRG_{g_j,b_j} \left(\bfL^{(j)}\right)
\]
where $\bfL^{(j)}$ denotes the lengths corresponding to the boundary components of $\Sigma_j$.

These slices fit together into a larger moduli space $\MRG(S_g \setminus \agamma)$ which can be topologized through its natural embedding into a product of combinatorial moduli spaces with variable boundary lengths. Denote by $\Delta \subseteq \mathbb{R}_{>0}^k$ the open simplex
\[
\Delta := \{\bfL \in \mathbb{R}_{>0}^k |  L_1 + \dots + L_k = 1\}
\]
and let $\MRG(S_g \setminus \agamma;\Delta)$ denote the total space of the fibration over $\Delta$ whose fiber above $\bfL \in \Delta$ is $\MRG(S_g \setminus \agamma; \bfL)$; this can be thought of as the ``projectivization'' of the full moduli space $\MRG(S_g \setminus \agamma)$ under the natural rescaling action.

For any marked complex structure $X \in \T_g$, the critical graph of $\mathrm{JS}(X,\gamma)$ is the metric ribbon graph obtained as the union of the critical leaves of $\Im(q) \in \mathcal{MF}_g$ endowed with the restriction of the singular flat metric induced by $q$ on $X$. Notice this graph has one connected component for each component of $S_g \setminus \vg$. Using the labeling of these components one can define the {\em critical graph map}
\[
\Xi(\cdot,\agamma) \colon \mathcal{T}_g \to \mathcal{MRG}(S_g \setminus \vg).
\]
Furthermore, after rescaling the metric of the critical graph $\Xi(X,\vg) \in \mathcal{MRG}(S_g \setminus \vg)$ by $1/\mathrm{Ext}_X(\gamma)$ one obtains the unit length critical graph map
\[
\Xi^1(\cdot,\vg) \colon \mathcal{T}_g \to \mathcal{MRG}(S_g \setminus \vg;\Delta).
\]
Since $\Stab_{0}(\vg)$ preserves labelings of complementary subsurfaces, the maps $\Xi(\cdot,\vg)$ and $\Xi^1(\cdot,\vg)$ are $\mathrm{Stab}_0(\vg)$-invariant.

\section{Near the multiple zero locus}
\label{sec:msl}

\subsection*{Outline of this section.} In this section we show that the Jenkins--Strebel differentials of most simple closed multi-curves lie deep within the principal stratum; see Proposition \ref{prop:key_estimate} for a precise statement. This control is crucial to apply the bounds on the variation of critical graphs proved in \S\ref{sec:weightsvaryA}; compare to Proposition \ref{prop:varA}. The main tools used in this section are the theory of Delaunay triangulations of quadratic differentials developed in work of Masur and Smillie \cite{MS91} and dual train tracks to triangulations.

\subsection*{Triangulations of quadratic differentials.} For the rest of this paper fix $g \geq 2$ and let $S_g$ be a compact, connected, oriented surface of genus $g \geq 2$.
By a \textit{triangulation} of a quadratic differential $q$ we mean a triangulation of its underlying Riemann surface whose edges are saddle connections of $q$ (and hence whose vertices are zeros of $q$).
A triangulation of a quadratic differential $q$ is said to be \textit{$L$-bounded} for some $L > 0$ if its edges have flat length $\leq L$.

Given a marked quadratic differential $q \in \QoT$ and a triangulation $\Delta'$ of $q$, one can pull back $\Delta'$ via the marking map to obtain an isotopy class of triangulation $\Delta$ on $S_g$.
In particular, because we are not marking the zeros of $q$, the vertices of the triangulation $\Delta$ are not fixed.
Recall that the mapping class group $\Mod_g$ acts properly discontinuously on $\QoT$ and $\T_g$ by changing markings.
It also acts on the set of isotopy classes of triangulations of $S_g$ by applying the mapping class and the association described above is equivariant.

Recall that $\pi: \QoT \to \T_g$ denotes the forgetful map.

\begin{lemma}
	\label{lem:triang_compact}
	Let $\Delta$ be an isotopy class of triangulation of $S_g$, let $\mathcal{K} \subseteq \tt$ be a compact subset, and $L > 0$. Then the subset of marked quadratic differentials $q \in  \pi^{-1}(\mcg \cdot\mathcal{K})$ having an $L$-bounded triangulation $\Delta'$ that pulls back to $\Delta$ via the marking of $q$ has compact closure in $\QoT$.
\end{lemma}

\begin{proof}
	Let $q \in  \pi^{-1}(\mcg \cdot\mathcal{K})$ and $\Delta'$ be an $L$-bounded triangulation of $q$ which pulls back to $\Delta$ via the marking of $q$. Fix a simple closed curve $\alpha$ on $S_g$. As $\Delta'$ is $L$-bounded and pulls back to $\Delta$ via the marking of $q$, one can bound the flat length $\ell_\alpha(q)$ of any geodesic representatives of $\alpha$ on $q$ uniformly in terms $\alpha$, $\chi(S_g)$, and $L$.
 This together with the fact that $q \in  \pi^{-1}(\mcg \cdot\mathcal{K})$ implies the hyperbolic length $\ell_\alpha(\pi(q))$ of the unique geodesic representative of $\alpha$ with respect to the marked hyperbolic structure on $S_g$ induced by $\pi(q) \in \mathcal{T}_g$ via uniformization can be bounded uniformly in terms of $\alpha$, $\chi(S_g)$, $L$, and $\mathcal{K}$. As the bundle $\pi \colon \QoT \to \tt$ has compact fibers and as the only was of escaping to infinity in $\mathcal{T}_g$ is to develop a simple closed curve of unbounded hyperbolic length, this finishes the proof.
\end{proof}

\subsection*{Delaunay triangulations of quadratic differentials.} For every quadratic differential $q$ denote by $\ell_{\min}(q)$ the length of its shortest saddle connections and by $\mathrm{diam}(q)$ its diameter. Every quadratic differential $(X,q)$ admits a triangulation by saddle connections which is \textit{Delaunay} with respect to the singularities of $q$ and the singular flat metric induced by $q$ on $X$ \cite[\S 4]{MS91}. We refer to any such triangulation as a \textit{Delaunay triangulation} of $q$. For the purposes of this discussion we will not need to appeal to the explicit construction of these triangulations. Rather, it will suffice to know they exist and satisfy the following properties.

\begin{lemma}
	\cite[Lemma 3.11]{ABEM12} 
	\label{lem:delaunay}
	Let $q \in \QoT$ and $\Delta$ be a Delaunay triangulation of $q$. Let $\gamma$ be a saddle connection of $q$.
	\begin{enumerate}
		\item If $\gamma$ belongs to $\Delta$, then $\ell_{\gamma}(q) \preceq_g \mathrm{diam}(q)$.
		\item If $\ell_{\gamma}(q) \leq \sqrt{2} \cdot \ell_{\min}(q)$, then $\gamma$ belongs to $\Delta$.
	\end{enumerate}
\end{lemma}

Directly from Lemmas \ref{lem:triang_compact} and \ref{lem:delaunay} we deduce the following result.

\begin{proposition}
	\label{prop:delaunay_comp}
	For every compact subset $\mathcal{K} \subseteq \mathcal{T}_g$ there exists a constant $L = L(\mathcal{K}) > 0$ and a finite collection of isotopy classes of triangulations $\{\Delta_i\}_{i=1}^n$ on $S_g$ with the following property. Let $q \in \pi^{-1}(\mathcal{K})$ and $\gamma$ be a saddle connection of $q$ attaining the minimal flat length among saddle connections of $q$. Then, there exists $i \in \{1,\dots,n\}$ and an $L$-bounded triangulation $\Delta'$ of $q$ having $\gamma$ as one of its edges and which pulls back to $\Delta_i$ via the marking of $q$.
\end{proposition}

\begin{proof}
Fix a compact subset $\mathcal{K} \subseteq \mathcal{T}_g$. 
By Lemma \ref{lem:delaunay}.1, there exists a constant $L = L(\mathcal{K}) > 0$ such that every Delaunay triangulation of a quadratic differential in $\pi^{-1}(\mathcal{K})$ is $L$-bounded. 
Moreover, by Lemma \ref{lem:delaunay}.2, these triangulations always contain the shortest saddle connections of $q$.
Thus, it suffices to show these triangulations pull back to finitely many triangulations on $S_g$.
This follows because triangulations of quadratic differentials in $\QoT$ have at most $4g-4$ vertices, so up to the action of the mapping class group there are finitely many combinatorial types of triangulations on $S_g$ that can arise.
Applying Lemma \ref{lem:triang_compact} and the proper discontinuity of the $\mcg$ action on $\QoT$, we see that there are only are finitely many isotopy classes of triangulations on $S_g$ that are $L$-bounded on some $q \in \pi^{-1}(\mathcal{K})$. This finishes the proof.
\end{proof}

\begin{remark}
We note that this statement is false if we mark the zeros of $q$ as well as the underlying surface.
The zeros can braid around each other while $q$ remains in a compact subset of $\QoT$, yielding infinitely many distinct triangulations that differ by the surface braid group.
This reflects the fact that the intersection of $\pi^{-1}(\mathcal K)$ with any (non-minimal) stratum is not compact.
\end{remark}

\subsection*{Train track coordinates.} We now discuss some aspects of the theory of train track coordinates. For more details we refer the reader to \cite{PH92}. A \textit{train track} $\tau$ on $S_g$ is an embedded $1$-complex satisfying the following conditions:
\begin{enumerate}
	\item Each edge of $\tau$ is a smooth path with a well defined tangent vector at each endpoint. All edges at a given vertex are tangent.
	\item For each component $R$ of $S_g \setminus \tau$, the double of $R$ along the smooth part of the boundary $\partial R$ has negative Euler characteristic.
\end{enumerate}
The vertices of $\tau$ where three or more edges meet are called \textit{switches}. By considering the inward pointing tangent vectors of the edges incident to a switch, one can divide these edges into \textit{incoming} and \textit{outgoing} edges. A train track $\tau$ on $S_g$ is said to be \textit{maximal} if all the components of $S_g \setminus \tau$ are trigons, i.e., the interior of a disc with three non-smooth points on its boundary.

A singular measured foliation $\lambda \in \MF_g$ is said to be {\em carried} by a train track $\tau$ on $S_g$ if it can be obtained by collapsing the complementary regions in $S_g$ of a measured foliation of a tubular neighborhood of $\tau$ whose leaves run parallel to the edges of $\tau$. In this situation, the invariant transverse measure of $\lambda$ corresponds to a counting measure $v$ on the edges of $\tau$ satisfying the \textit{switch conditions}: at every switch of $\tau$ the sum of the measures of the incoming edges equals the sum of the measures of the outgoing edges. Every $\lambda \in \MF_g$ is carried by some maximal train track $\tau$ on $S_g$.

Given a maximal train track $\tau$ on $S_g$, denote by $V(\tau) \subseteq \smash{(\mathbf{R}_{\geq0})^{18g-18}}$ the $6g-6$ dimensional closed cone of non-negative counting measures on the edges of $\tau$ satisfying the switch conditions. The set $V(\tau)$ can be identified with the closed cone $U(\tau) \subseteq \MF_g$ of singular measured foliations carried by $\tau$. These identifications give rise to coordinates on $\MF_g$ called \textit{train track coordinates}. The transition maps of these coordinates are piecewise integral linear. In particular, $\MF_g$ can be endowed with a natural $6g-6$ dimensional piecewise integral linear structure where the integral points $\MF_g(\mathbb Z)$ correspond to integrally-weighted simple closed multi-curves.

\subsection*{Train tracks dual to triangulations.} 
We now use our discussion of triangulations of $q \in \pi^{-1}(\mathcal K)$ to control what the horizontal foliations can coarsely look like.
We begin by recalling the construction of a train track dual to a triangulation; compare 
\cite{Mir08a}, \cite{Ara22}, \cite{CF}, and \cite{CF2}.

Let $\Delta$ be an isotopy class of triangulation on $S_g$. On each of the triangles of $\Delta$ consider a $1$-complex as in Figure \ref{fig:complex_1}; the edges of this complex that do not intersect the sides of the triangle will be referred to as \textit{inner edges}. Join these complexes along the edges of $\Delta$ as in Figure \ref{fig:complex_2} to obtain a complex on $S_g$. We say that a train track $\tau$ on $S_g$ is {\em dual} to $\Delta$ if it can be obtained from this complex by deleting one inner edge in each triangle of $\Delta$.
We remark that this operation is well-defined even though we are only considering isotopy classes: two isotopic triangulations will give isotopic collections of dual train tracks.

	\begin{figure}[h]
	\centering
	\begin{subfigure}[b]{0.4\textwidth}
		\centering
		\includegraphics[width=0.7\textwidth]{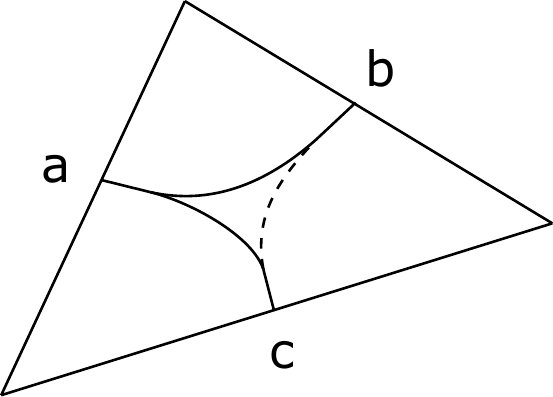}
		\caption{The dual $1$-complex and dual train track in a triangle.}
		\label{fig:complex_1}
	\end{subfigure}
	\quad \quad \quad
	~ 
	\begin{subfigure}[b]{0.4\textwidth}
		\centering
		\includegraphics[width=0.7\textwidth]{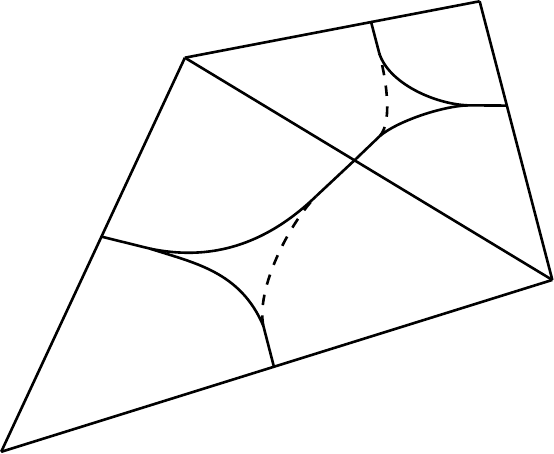}
		\caption{Joining the dual $1$-complexes and train tracks.}
		\label{fig:complex_2}
	\end{subfigure}
	\caption{The $1$-complex and the train track dual to a triangulation. The train track is obtained by removing the dashed edges from the 1-complex.} 
	\label{fig:complex}
	\end{figure}

Let $q \in \QoT$ and $\Delta'$ be a triangulation of $q$. Denote by $\Delta$ the triangulation of $S_g$ obtained by pulling back $\Delta'$ via the marking of $q$. The horizontal foliation $\Im(q) \in \MF_g$ is carried by a train track dual to $\Delta$. Indeed, let $T'$ be a triangle of $\Delta'$. Label the edges of $T'$ by $a,b,c$ so that
\[
\int_a \thinspace d\Im(q) = \int_b \thinspace d\Im(q) + \int_c \thinspace d\Im(q)
\]
This labelling is unique unless one of the edges of $T'$ is horizontal, in which case there exist two such labelings. On $T'$ consider a  $1$-complex as in Figure \ref{fig:complex_1} and delete the inner edge of this complex opposite to $a$. Consider the corresponding $1$-complexes on all the triangles of $\Delta'$.
Joining these complexes along the edges of $\Delta'$ as in Figure \ref{fig:complex_2} and pulling back the resulting complex to $S_g$ yields a train track $\tau$ dual to $\Delta$.
This train track carries $\Im(q) \in \MF_g$ and, for each edge $e$ of $\Delta'$, the counting measure of the coresponding edge of $\tau$ is equal to $\int_e \thinspace d\Im(q)$. If $q$ is in the principal stratum of quadratic differentials, the train track $\tau$ obtained through this construction is maximal, and in general, the singularity structure of $q$ is reflected in the combinatorics of $\tau$.

Directly from the discussion above and Proposition \ref{prop:delaunay_comp} we deduce the following, which is analogous to both \cite[Proposition 2.7]{Ara22} and \cite[Section 10]{CF}.

\begin{proposition}
	\label{lem:finite_tt}
	For every compact subset $\mathcal{K} \subseteq \mathcal{T}_g$ there exists a finite collection of train tracks $\{\tau_i\}_{i=1}^n$ on $S_g$ with the following property. Let $q \in \pi^{-1}(\mathcal{K})$ and $\gamma$ be a saddle connection of $q$ that attains the minimal flat length among saddle connections of $q$. Then, there exists $i \in \{1,\dots,n\}$ such that $\tau_i$ carries $\Im(q)$ and such that the corresponding counting measure on the edges of $\tau_i$ gives weight $\int_{\gamma} \thinspace d\Im(q)$ to one of the edges of $\tau_i$.
\end{proposition}

\subsection*{The key estimate.} We are now ready to prove the main estimate of this section. This estimate will allow us control the equidistribution problem of interest near the multiple zero locus in a sufficiently strong way.

Let $\JS \colon \T_g \times \MF_g \to \QoT$ denote the Hubbard--Masur map from Theorem \ref{theo:HM1} which to every $(X,\lambda) \in \T_g \times \MF_g$ assigns the unique marked quadratic differential $q \in \pi^{-1}(X)$ such that $\Im(q) = \lambda$.
This differential has area $\Ext_X(\lambda)$, so rescaling it by the square root of extremal length results in a unit area differential.
Let $K_\sigma \subseteq \QoT$ denote the set of marked unit area quadratic differentials $q \in \QoT$ belonging to the principal stratum with $\ell_{\min}(q) \geq \sigma$.
Given $X \in \mathcal{T}_g$ consider the set
\[
F(X,\sigma) := \left\{ \alpha \in \mathcal{MF}_g(\mathbb{Z}) \ \middle| \ 
\frac{1}{\sqrt{\Ext_X(\alpha)}} \JS(X,\alpha) \notin K_\sigma \right\}
\]
and for any compact set $\mathcal K \subset \T_g$ define $F(\mathcal K, \sigma) := \bigcup_{X \in \mathcal K} F(X, \sigma)$.

\begin{proposition}
    \label{prop:key_estimate}
    Let $\mathcal{K} \subseteq \mathcal{T}_g$ be a compact subset. Then there exists a constant $C > 0$ such that for every $X \in \mathcal{K}$, every $\sigma > 0$, and every $L > 1$,
    \[
    \#\{\alpha \in F( \mathcal K,\sigma) \ | \ \mathrm{Ext}_\alpha(X) \leq L^2 \} \leq C \cdot L^{6g-7} + C \cdot \sigma \cdot L^{6g-6}.
    \]
\end{proposition}

\begin{proof}
Let $\{\tau_i\}_{i=1}^n$ be the finite collection of train tracks on $S_g$ provided by Proposition \ref{lem:finite_tt}.
    Now consider $\alpha \in \mathcal{MF}_g(\mathbb{Z})$ such that $\mathrm{Ext}_\alpha(X) \leq L^2$ and
    \[q := \mathrm{JS}(X,\alpha/ \sqrt{\mathrm{Ext}_\alpha(X)} ) \in \QoT \setminus K_\sigma.\]
    Suppose first that $q$ is not in the principal stratum. Then, by construction, the train track $\tau_i$ carrying $\alpha = \sqrt{\mathrm{Ext}_\alpha(X)} \cdot \Im(q)$ is not maximal. Suppose now that $q$ is in the principal stratum and that $\gamma$ is a saddle connection of $q$ attaining $\ell_{\min}(q)$. It follows by construction that, when carried by the corresponding $\tau_i$, one of the resulting weights of $\alpha$ is at most $\sigma \cdot \sqrt{\mathrm{Ext}_\alpha(X)} \leq \sigma \cdot L$. As the collection of train tracks $\{\tau_i\}_{i=1}^n$ is finite, the desired bound follows from a standard lattice point counting argument.
\end{proof}

\section{Uniform geometric estimates}
\label{sec:weightsvaryA}

\subsection*{Outline of this section.} In this section we show that the weights of the metric ribbon graph 
$\Xi^1(X, \vec{\alpha})$ vary uniformly over $\alpha$ as $X$ varies in a suitably chosen neighborhood of moduli space; see Proposition \ref{prop:varA} for a precise statement. The proof is based on a uniform continuity argument and uses the AGY metric on strata of quadratic differentials introduced by Avila, Gouëzel, and Yoccoz \cite{AGY06} in a crucial way. Before proving Proposition \ref{prop:varA} we discuss some of the basic properties of the AGY metric following \cite{AGY06} and \cite{AG13}.

\subsection*{The AGY metric.} Let $\mathcal{Q} \subseteq \mathcal{QT}_g$ be the principal stratum of marked quadratic differentials on $S_g$, that is, the subset of differentials with only simple zeros. Denote points in $\mathcal{Q}$ by $(X,q)$, where $X$ is a marked Riemann surface and $q$ is a quadratic differential on $X$. Let $(Z,\omega) \to (X,q)$ be the holonomy double cover of $(X,q)$. The tangent space of $\mathcal{Q}$ at $(X,q)$ can be identified with the relative cohomology group $H^1_\mathrm{odd}(Z,\Sigma;\mathbf{C})$ where $\Sigma \subseteq Z$ denotes the set of zeroes of $\omega$, and the subscript odd denotes the $-1$ eigenspace of the canonical involution of $Z \to X$. For $v \in H^1_\mathrm{odd}(Z,\Sigma;\mathbf{C})$ consider the norm
\[
\|v\|_{q} := \sup_{s \in \mathcal{S}} \left| \frac{v(s)}{\mathrm{hol}_\omega(s)} \right|,
\]
where $\mathcal{S}$ denotes the set of saddle connections of $\omega$ and $\mathrm{hol}_\omega(s) \in \mathbf{C}$ denotes the holonomy of the saddle connection $s$ with respect to $\omega$; we will sometimes denote this holonomy by $|s|_\omega$. By work of Avila, Gouëzel, and Yoccoz \cite{AGY06}, this definition indeed gives rise to a norm on $H^1_\mathrm{odd}(Z,\Sigma;\mathbf{C})$ and the corresponding Finsler metric on $\mathcal{Q}$ is complete.
We refer to this metric as the AGY metric of $\mathcal{Q}$ and denote it by $d_{\AGY}$.
For any $\mathcal{C}^1$ path $\kappa \colon [0,1] \to \mathcal{Q}$, denote its AGY length by
\[
\mathrm{length}(\kappa) := \int_0^1 \|\kappa'(t)\|_{\kappa(t)} \thinspace dt.
\]

We observe that the rescaling action $q \mapsto rq$ for $r \in \RR_{>0}$ is an isometry of the AGY metric.
Indeed, given $v \in H^1_\mathrm{odd}(Z,\Sigma;\mathbf{C})$ we have that
$\|rv\|_{rq} = \|v\|_q$, and so given any $C^1$ path $\kappa(t)$ from $q$ to $q'$ we can compute the length of the corresponding path $r\kappa(t)$ from $rq$ to $rq'$:
\[
\mathrm{length}(r\kappa) 
= \int_0^1 \|r\kappa'(t)\|_{r\kappa(t)} \thinspace dt
= \int_0^1 \|\kappa'(t)\|_{\kappa(t)} \thinspace dt
= \mathrm{length}(\kappa).
\]

\subsection*{The fundamental property.}
Fix $(X,q) \in \mathcal{Q}$ and let $\Psi = \Psi_q$ be the local parametrization of $\mathcal{Q}$ by its tangent plane $H^1_\mathrm{odd}(Z,\Sigma;\mathbf{C})$ as above. More formally, define $\Psi(v)$ for $v \in H^1_\mathrm{odd}(Z,\Sigma;\mathbf{C})$ as follows. Consider the path $\kappa$ starting from $(X,q)$ with $\kappa'(t) = v$ for all times $t$. For small $t$ the path $\kappa(t)$ is well defined. It is possible that $\kappa(t)$ could be not defined for large $t$. If the path $\kappa$ is well defined for all $t \in [0,1]$ define $\Psi(v) := \kappa(1) \in \mathcal{Q}$.

Denote by $B(0,r)$ the ball of radius $r$ centered at the origin of $H^1_\mathrm{odd}(Z,\Sigma;\mathbf{C})$ with respect to the norm $\|\cdot\|_q$. The following results, which will be crucial in later sections, give control on the image of the ``exponential map'' $\Psi$.

\begin{proposition}
\label{prop:fund}
The map $\Psi$ is well defined on $B(0,1/2)$. Moreover,
\begin{enumerate}
    \item\label{item:expoutbd} For $v \in B(0,1/2)$,
\[
d_\mathrm{AGY}(q,\Psi(v)) \leq 2 \cdot \|v\|_q.
\]
\item\label{item:expnormbd} For every $v \in B(0,1/2)$ and every $w \in H^1_\mathrm{odd}(Z,\Sigma;\mathbf{C})$,
\[
1/2 \leq \frac{\|w\|_q}{\|w\|_{\Psi(v)}} \leq 2.
\]
\item\label{item:expinbd} For $v \in B(0,1/25)$,
\[
d_\mathrm{AGY}(q,\Psi(v)) \geq \|v\|_q/2.
\]
In particular, any $q' \in \cQ$ with $d_{\AGY}(q, q') < 1/50$ is equal to $\Psi(v)$ for some $v \in B(0, 1/25)$.
\end{enumerate}
\end{proposition}

Apart from the last claim of (3), this is just the statement of \cite[Proposition 5.3]{AG13} generalized to arbitrary vectors; we provide an outline of the final claim.

\begin{proof}[Proof sketch]
Consider the image of $B(0, 1/25)$ under $\Psi$.
By invariance of domain, this is open; in particular, it contains an open AGY ball $B_{\AGY}$ of some maximal radius $r>0$.
Now by the first claim of (3), we have that $\Psi^{-1}(B_{\AGY})$ is contained in $B(0, 2r)$.
In particular, if $r < 1/50$ then we can find a slightly larger ball $B(0, 2r + \varepsilon)$ whose $\Psi$ image strictly contains $B_{\AGY}$ (by the first claim of (3) again), hence $r$ was not maximal to begin with. So $r \ge 1/50$.
\end{proof}

A proof of the remaining statements can be obtained by closely following the arguments of \cite{AG13}.
Indeed, the key input in the proof of Proposition \ref{prop:fund} is \cite[Proposition 5.5]{AG13}, which gives a bound on the growth of the parallel translate of a vector in terms of the length of the path along which it is translated, and is true for any $v$, not just those lying in the unstable distribution.

We record for future use a similar result on how the length of saddle connections changes along paths; compare with \cite[Lemma 5.6]{AG13}.

\begin{lemma}\label{lem:saddlechange}
Suppose that $\kappa:[0,1] \to \mathcal{Q}$ is a $C^1$ path and that $s$ is a saddle connection that survives in $\kappa(t)$ for all $t \in [0,1]$. Then 
\[
e^{-\mathrm{length}(\kappa)}
\le
\frac{|s|_{\kappa(0)}}{|s|_{\kappa(1)}} 
\le
e^{\mathrm{length}(\kappa)}.
\]
\end{lemma}

We remark that only the upper bound is stated in \cite{AG13}, but the proof there also yields the desired lower bound.

\subsection*{Variation of weights.}
With these preliminaries taken care of, we can now show that the geometry of the horizontal ribbon graph is controlled uniformly over most simple closed multi-curves $\alpha$ as $X$ varies in a small ball in $\T_g$.
Compare with Lemma 4.1 and Proposition 4.2 of \cite{spine}.

Recall that if $\vec{\alpha}$ is an ordered, oriented simple closed multi-curve on $S_g$, then $\alpha \in \mathcal{MF}_g(\mathbb{Z})$ denotes its equivalence class as a singular measured foliation. Recall that for a compact set $\mathcal K \subseteq \T_g$ and $\sigma > 0$, the set $F(\mathcal K, \sigma)$ denotes those integral simple closed multi-curves $\alpha \in \mathcal{MF}_g(\mathbb{Z})$ so that, for some $X \in \mathcal K$, the differential $\JS(X, \alpha)$ is either not in $\cQ$ or has a saddle connection shorter than $\sigma \sqrt{\Ext}_X(\alpha)$. 

\begin{proposition}
\label{prop:varA}
Fix a compact subset $\mathcal K \subseteq \T_g$. Then, for every $\varepsilon>0$ and every $\sigma >0$, there exists $\delta = \delta(\mathcal K, \varepsilon, \sigma) > 0$ so that for any ordered, oriented simple closed multi-curve $\vec{\alpha}$ on $S_g$ such that $\alpha \notin F(\mathcal K, \sigma)$ and any two $X, X' \in \mathcal K$ with $d_{\Teich}(X, X') \le \delta$, the following hold.
\begin{enumerate}
    \item The horizontal ribbon graphs $\Sp(X, \vec{\alpha})$ and $\Sp(X', \vec{\alpha})$ have the same topological type, i.e., live in the same facet of $\MRG(S_g \setminus \vec{\alpha})$.
    \item For every edge $e$ of $\Sp(X, \vec{\alpha})$ and/or $\Sp(X', \vec{\alpha})$, we have that
\[e^{-\varepsilon} \le \frac{|e|_{\Sp^1(X, \vec{\alpha})}}{|e|_{\Sp^1(X', \vec{\alpha})}} \le e^{\varepsilon}.\]
\end{enumerate}
\end{proposition}

The proof of this Proposition has two steps: we first show that the hypotheses imply that the differentials $\JS(X, \alpha)$ and $\JS(X', \alpha)$ are close in the AGY metric, then use properties of the metric to conclude the desired statements.

The first step is accomplished using the uniform continuity of the Hubbard--Masur map.
To get this, we need to restrict to a compact set.
Cutting out $K_\sigma$ allows us to avoid a neighborhood of the multiple zero locus, but we must also impose a bound on the areas of differentials under consideration. 
As such, we first record an {\em a priori} bound on the extremal length of foliations as the base surface ranges over a compact set.
The following is a direct consequence of Theorem \ref{thm:Kerckhoff}.

\begin{lemma}\label{lem:Extcomp}
For every compact $\mathcal K \subseteq \T_g$, there is a $c_{\mathcal K} > 1$ such that for every $\lambda \in \MF_g$ and every $X, X' \in \mathcal K$,
\[
c_{\mathcal K}^{-1} \le 
\frac{\sqrt{\Ext}_{X'}(\lambda)}{\sqrt{\Ext}_{X}(\lambda)}
\le c_{\mathcal K}.
\]
\end{lemma}

We now show that if $X$ and $X'$ are close in the Teichm{\"u}ller metric, then most of their Jenkins--Strebel differentials are close in the AGY metric.

\begin{lemma}\label{lem:nbhd}
Fix $\mathcal K$, $\sigma$ as in Proposition \ref{prop:varA}.
For any $\zeta>0$, there is a $\delta>0$ so that for any $X, X' \in \mathcal K$ with $d_{\Teich}(X, X') < \delta$ and any $\alpha \notin F(\mathcal K, \sigma)$,
\[d_{\AGY}(\JS(X, \alpha), \JS(X', \alpha)) < \zeta.\]
\end{lemma}
\begin{proof}
Fix any metric $d_{\mathcal{MF}_g}$ on $\MF_g$ equipping it with the standard topology. By Theorem \ref{theo:HM1}, the map $\JS: \T_g \times \MF_g \to \QT$ is a homeomorphism, and is hence uniformly continuous on the compact set
\[J:=(\mathcal K \times \MF_g) \cap 
\Ext^{-1}([c_{\mathcal K}^{-2}, c_{\mathcal K}^2]) \cap 
\JS^{-1}(\mathbb{R}_{>0} \cdot K_\sigma).\]
That is, for any $\zeta >0$ there is a $\delta >0$ so that for
any $(X,\lambda), (X', \lambda) \in J$ with 
\[d_{\Teich}(X, X') < \delta \text{ and } d_{\MF_g}(\lambda, \lambda') < \delta\]
then we have that
\[d_{\AGY}(\JS(X, \lambda), \JS(X', \lambda')) < \zeta.\]

In particular, Lemma \ref{lem:Extcomp} (alternatively, Theorem \ref{thm:Kerckhoff}) implies that
\[\sqrt{\Ext}_{X'}\left(\alpha/\sqrt{\Ext}_X(\alpha)\right) 
= \frac{\sqrt{\Ext}_{X'}(\alpha)}{\sqrt{\Ext}_{X}(\alpha)}
\in [c_{\mathcal K}^{-1}, c_{\mathcal K}]
\]
so setting $\lambda = \lambda' = \alpha/ \sqrt{\Ext}_X(\alpha)$ we have that $\JS(X, \lambda)$ and $\JS(X', \lambda)$ are $\zeta$ apart.
The statement for $\alpha$ itself follows from the fact that rescaling is an isometry of the AGY metric.
\end{proof}

Now since $\JS(X, \alpha)$ and $\JS(X', \alpha)$ are close and live in the same leaf of the unstable foliation, we can use the exponential map $\Psi$ to connect them via a path completely contained in the unstable leaf $\{\JS(X, \alpha) \ | X \in \mathcal{T}_g\}$. Analyzing this map yields the proof of the main result of this section.

\begin{proof}[Proof of Proposition \ref{prop:varA}]
We begin by observing that, by definition of the $\Xi^1$ map,
\[
\frac{|e|_{\Xi^1(X, \vec{\alpha})}}
{|e|_{\Xi^1(X', \vec{\alpha})}}
=
\frac{|e|_{\Xi(X, \vec{\alpha})}}
{|e|_{\Xi(X', \vec{\alpha})}}
\cdot 
\frac{\Ext_{X'}(\alpha)}
{ \Ext_X(\alpha)}.
\]
By Theorem \ref{thm:Kerckhoff}, the ratio of extremal lengths is bounded arbitrarily close to 1 (so long as we take $\delta$ small enough), so it suffices to compare the geometry of the non-normalized critical graphs.
For ease of notation, throughout the rest of the proof let us denote
\[q := \JS(X, \alpha) 
\text{ and }
q' := \JS(X', \alpha).\]

Fix $0 < \varepsilon < 2/25$.
Lemma \ref{lem:nbhd} tells us that by taking $\delta$ small enough we can ensure $q$ and $q'$ are $\varepsilon/4 < 1/50$ close in the AGY metric, so by Proposition \ref{prop:fund}.\ref{item:expinbd} we have that $q' = \Psi(v)$ for some $v \in B(0,\varepsilon/2)$.
But now we know that $q$ and $q'$ have the same horizontal foliation, hence the same vertical periods. Therefore we must have that $v$ is real, i.e.,
$v \in H^1_{\text{odd}}(Z, \Sigma; \mathbb{R})$
with respect to the natural splitting 
\[H^1_{\text{odd}}(Z, \Sigma; \mathbb{C}) \cong H^1_{\text{odd}}(Z, \Sigma; \mathbb{R}) \oplus H^1_{\text{odd}}(Z, \Sigma; i\mathbb{R}).\]

In particular, this implies that the period of every horizontal saddle remains real along $\kappa(t)$. Since the exponential map $\Psi$ is well-defined (and proper), we see that every horizontal saddle must persist along the entire path $\{\kappa(t)\}_{t=0}^1$ (otherwise some saddle would shrink to zero, but doing so leaves the stratum). Hence the topological type of the horizontal saddle connection graphs of $\kappa(0)=q$ and $\kappa(1)=q'$ are the same, establishing the first part of the Proposition.

For the second part, we note that $\|v\|_{\kappa(t)}$ grows by at most a factor of $2$ along the entire path (Proposition \ref{prop:fund}.\ref{item:expnormbd}), so the length of $\kappa$ is at most $\varepsilon$.
Since each horizontal saddle persists along the entire path, Lemma \ref{lem:saddlechange} implies that the length of each can change by a factor of at most $e^{\mathrm{length}(\kappa)}$.
Thus the lengths on $q$ and $q'$ of every edge of $\Sp(X, \vec{\alpha})$ have ratio bounded by $e^{\pm\varepsilon}$, proving the second claim. 
\end{proof}

\section{Horoball measures}
\label{sec:critical_equi}

\subsection*{Outline of this section.} In this section we introduce ``critical-JS-horoballs'' and show that, as they expand over moduli space, they equidistribute with respect to the Masur-Veech measure. See Proposition \ref{prop:unstable_horoball_equid} for a precise statement. Our preliminary discussion relies on work of Athreya, Bufetov, Eskin, and Mirzakhani \cite{ABEM12} as well as an expression for critical-JS-horoballs in terms of the fibered Kontsevich measure.
The equidistribution result is a consequence of work of Forni \cite{push}, which in turn relies on breakthroughs of Eskin, Mirzakhani, and Mohammadi \cite{EM,EMM}. Throughout we use the normalizations in \cite{ABEM12}, \cite{ABEMeff}, and \cite{spine} for all the measures considered.

\subsection*{The Masur-Veech measure.}
Let $\vg := (\vg_1,\dots,\vg_k)$ be an ordered, oriented simple closed multi-curve on $S_g$ with underlying multi-curve $\gamma \in \mathcal{MF}_g$. Recall that $\mathrm{Stab}_0(\vg) \subseteq \mathrm{Mod}_g$ denotes the oriented stabilizer of $\vg$, i.e. the set of mapping classes that fix each component of $\vg$ together with its orientation.

Recall that $\nu_{\mathrm{\mathrm{MV}}}$ denotes the Masur-Veech measure on $\mathcal{Q}^1\mathcal{T}_g$. The forgetful map $\pi: \mathcal{Q}^1\mathcal{T}_g \to \mathcal{T}_g$ pushes the Masur--Veech measure down to a Lebesgue-class measure $\bfm := \pi_* \nu_{\MV}$ on $\T_g$, which we also refer to as the Masur--Veech measure.
Both $\nu_\MV$ and $\bfm$ are $\Mod_g$-invariant. Let $\tnu_\MV$ and $\tbfm$ denote the corresponding local pushforwards to $\QoT / \Stab_0(\vg)$ and $\mathcal{T}_g/\mathrm{Stab}_0(\vg)$, and, similarly, let $\hnu_\MV$ and $\hbfm$ be the pushforwards of $\tnu_\MV$ and $\tbfm$ to $\QoM$ and $\mathcal{M}_g$. 
One could of course also define $\tbfm$ and $\hbfm$ by pushing forward $\tnu_\MV$ and $\hnu_\MV$ under the corresponding forgetful maps. Denote the total mass of $\widehat{\nu}_{\mathrm{\mathrm{MV}}}$, or, equivalently, $\widehat{\mathbf{m}}$, by $b_g > 0$.

Recall that for any Riemann surface $X$, we denote by $S(X)$ the sphere of unit area quadratic differentials on $X$. Let $s_X$ be the conditional probability measure on $S(X)$ induced by $\mathbf{m}$ on $\mathcal{T}_g$. It is characterized by the disintegration formula
\begin{equation}\label{eqn:MV_sphere_disint}
d \nu_{\MV}(X,q) = d s_X(q) \, d \bfm(X),
\end{equation}
together with similar expressions for the measures $\tnu_\MV$ and $\hnu_\MV$ on quotients.
\footnote{We note that we do not need to worry about sizes of stabilizers when recording disintegration formulas at the level of moduli space.
For $g \ge 3$, we have that $\bfm$-almost every $X \in \M_g$ has no nontrivial automorphisms and so the fiber of $\QoM$ over $X$ is the entire sphere $S(X)$. For $g = 2$ every surface is hyperelliptic, but so is every quadratic differential.}

\subsection*{The Hubbard-Masur function.} Recall that the singular measured foliation $\Im(q) \in \mathcal{MF}_g$ denotes the horizontal foliation of $q \in \mathcal{Q}^1\mathcal{T}_g$, and that $[\Im(q)] \in \mathcal{PMF}_g$ denotes its projective class. As observed in \cite{ABEM12}, every leaf of the unstable foliation $\Fol^{u}$ of $\mathcal{Q}^1\mathcal{T}_g$ carries a conditional measure that is uniformly expanded by the Teichm{\"u}ller geodesic flow.
This measure can be explicitly obtained as follows. By definition, any leaf of $\Fol^u$ is of the form
\[\Fol_{[\beta]}^u = \{ q \in \mathcal{Q}^1\mathcal{T}_g \mid [\Im(q)] = [\beta] \in \mathcal{PMF}_g\}\]
for some $\beta \in \MF_g$. By Theorem \ref{theo:HM1},
this set can in turn be identified with the open set $\MF_g(\beta) \subseteq \MF_g$ of singular measured foliations on $S_g$ that together with $\beta$ bind the surface. The Thurston measure $\mu_{\mathrm{Thu}}$ on $\MF_g$ therefore restricts to a non-trivial measure on $\MF_g(\beta)$ and hence gives rise to a measure on $\Fol_{[\beta]}^u$ which we denote by $\nu_{u,\beta}$.

Recall that $\{g_t \colon \mathcal{Q}^1\mathcal{T}_g \to \mathcal{Q}^1\mathcal{T}_g\}_{t \in \mathbb{R}}$ is the Teichmüller geodesic flow on $\mathcal{Q}^1\mathcal{T}_g$. The fundamental scaling property described by the formula
\begin{equation}\label{eqn:scaling}
 (g_t)_* \nu_{u,\beta} = e^{-(6g-6)t} \nu_{u,\beta}   
\end{equation}
is then a consequence of the fact that the Teichm{\"u}ller geodesic flow stretches the horizontal direction, expanding the measure on the vertical foliation, and thus acting by multiplication by $e^t$ on $\MF_g(\beta)$.

By Theorems \ref{theo:HM1} and \ref{theo:GM},
the forgetful map $\pi:\QoT \to \T_g$ restricts to a homeomorphism $\pi_{[\beta]}$ between $\smash{\Fol^u_{[\beta]} \cong \MF(\beta)}$ and $\T_g$.
The pushforward of $\nu_{u, \beta}$ by this map is in the Lebesgue measure class. 
Following \cite{ABEM12}, we denote by
\[\lambda^+(q) :=
\frac{d \bfm}{d (\pi_{[\beta]})_* \nu_{u,\beta}}(X)\]
the corresponding Radon-Nikodym derivative, where $\smash{q = \pi^{-1}_{[\beta]}(X)}$, that is, $q =\JS(X, \beta)/ \sqrt{\Ext}_X(\beta)$.
We recall from \cite{ABEM12} that the {\em Hubbard--Masur function} is defined to be the following integral:
\[\Lambda(X) := \int_{Q^1(X)} \lambda^+(q) \thinspace d s_X.\]
Directly from the definitions, we observe that both $\lambda^+(q)$ and $\Lambda(X)$ are $\Mod_g$-invariant; see also \cite[top of page 1063]{ABEM12}.

Tracing through the definitions, one can arrive at the following formulation.

\begin{lemma}\label{lem:HMfunc_Ext}
\cite[Proposition 2.3 (iii)]{ABEM12}
For any $X \in \T_g$, $$\Lambda(X) = \mu_{\mathrm{Thu}} \left( \{\beta \in \MF \mid \Ext_{X}(\beta) \le 1\} \right).$$
\end{lemma}

In fact, Mirzakhani proved the following strong result.

\begin{theorem}\label{thm:HMconst}
\cite[Theorem 5.10]{Dumas_skin}
The function $X \in \mathcal{T}_g \mapsto \Lambda(X)$ is constant.
\end{theorem}

As such, we refer to this value as the {\em Hubbard--Masur constant} $\Lambda_g > 0.$

\subsection*{Horoballs}
For every $L>0$ we define the (total) extremal length horoball measure $\bfm_{\vg}^L$ on $\T_g$ by restricting $\bfm$ to the set of marked Riemann surfaces $X \in \mathcal{T}_g$ on which 
$\sqrt{\Ext}_X(\gamma)  \le L.$
We similarly define the (total) unstable horoball measure $\smash{\nu_{u,\vg}^L}$ on $\QoT$ by restricting $\nu_{u, \vg}$ to the preimage of this set under the forgetful map $\pi \colon \mathcal{Q}^1\mathcal{T}_g \to \mathcal{T}_g$.
The extremal length of $\gamma$ on $X$ is the same as the area of the quadratic differential $q= \JS(X,\gamma)$, which in turn is equal to the geometric intersection number of the horizontal and vertical foliations of $q$, so $\smash{\nu_{u,\vg}^L}$ can equivalently be defined by restricting the Thurston measure on $\MF_g(\gamma)$ to  the set
\[ \{\beta \in \MF_g(\gamma) \mid i(\beta, \gamma) \le L\}\]
and pushing this ``Thurston horoball measure'' $\mu_{\Th}^L$ forward to $\Fol^u_{[\gamma]}$.
We then take local pushforwards to get measures
\[\begin{array}{ccrl}
\widetilde{\bfm}_{\vg}^L     & \text{ on } &\T_g / \Stab_0(\vg), & \text{ } \\
\tnu_{u,\vg}^L     & \text{ on } &\QoT / \Stab_0(\vg), \\
\tmu_\Th^L            & \text{ on } &\MF(\gamma) / \Stab_0(\vg). \\
\end{array}\]

One can check that $\tmu_\Th^L$ is finite, as the usual Thurston measure is locally finite on $\MF_g$ and the closure (inside $\MF_g$, including the $0$ foliation) of a fundamental domain for the action of $\Stab_0(\vg)$ on the support of $\mu_\Th^L$ is compact. 
This implies that $\tnu_{u,\gamma}^L$ and $\widetilde{\bfm}_{\gamma}^L$ are also finite, so we can take the (global) pushforwards of $\tnu_{u,\gamma}^L$ and $\widetilde{\bfm}_{\gamma}^L$ to the moduli spaces $\QoM$ and $\mathcal{M}_g$; denote the resulting measures by $\hnu_{u,\gamma}^L$ and $\widehat{\mathbf{m}}_\gamma^L$.

Recall from \S\ref{sec:prelim} that $\MRG(S_g \setminus \vg; \Delta)$ denotes the moduli spaces of ribbon graphs of complementary subsurfaces to $\vg$ of total boundary length $2$ and matching boundary lengths along the components of $\vg$. Recall also that $\Sp^1(X, \vg) \in \MRG(S_g \setminus \vg;\Delta)$ denotes the critical graph of the Jenkins-Strebel differential $\mathrm{JS}(X,\gamma)$ rescaled so that the boundaries have total length $2$.

We also want to consider subsets of the horoballs above by conditioning on the shape of the horizontal separatrices of the corresponding Jenkins--Strebel differentials.
To this end, for any $L>0$ and any non-zero, continuous, compactly supported function
$h: \MRG(S_g \setminus \vg; \Delta) \to \RR$, define the $\Xi^1$-horoball measure on $\T_g$ by
\begin{equation}
\label{eq:Xihoro_meas}
d\bfm_{\agamma,h}^L(X) := \mathbbm{1}_{[0,L]}\left(\sqrt{\Ext}_X(\gamma)\right) h\left(\Xi^1(X, \vg) \right) d\bfm(X).
\end{equation}
We similarly define a version supported on the unstable leaf corresponding to $\gamma$:
\begin{equation}
\label{eq:Xihoro_meas_lift}
d\nu_{u,\agamma,h}^L(X,q) := \mathbbm{1}_{[0,L]}\left(\sqrt{\Ext}_X(\gamma)\right)  h\left(\Xi^1(X, \vg) \right)  d\nu_{u,\gamma}(q).
\end{equation}
We informally refer to the measures $\smash{\bfm_{\agamma,h}^L}$ as ``critical-JS-horoballs". Compare with the definition of ``RSC-horoballs'' from \cite[Equation (5.1)]{spine}. Notice that the measures $\smash{\bfm_{\agamma,h}^L}$ are not equal to the pushforwards of the measures $\smash{\nu_{u,\agamma,h}^L}$ under the Hubbard--Masur map: they differ by the Hubbard--Masur function.

The measures $\smash{\bfm_{\agamma,h}^L(X)}$ and  $\smash{\nu_{u,\agamma,h}^L(q)}$ are $\Stab_0(\vg)$-invariant and so as in the case of the total horoball measures we can take their local pushforwards $\smash{\widetilde{\bfm}_{\agamma,h}^L}$ and $\smash{\tnu_{u,\agamma,h}^L}$ after quotienting by $\Stab_0(\vg)$.
We can then further push each down to finite measures 
$\smash{\widehat{\bfm}_{\agamma,h}^L}$ and $\smash{\widehat \nu_{u,\agamma,h}^L}$
on $\M_g$ and $\QoM$, respectively.

\subsection*{The fibered Kontsevich measure} 
In the next two subsections, we discuss how the ``fibered Kontsevich measure'' describes critical-JS-horoballs in terms of the combinatorial data of metric ribbon graphs; see Proposition \ref{prop:pushTh_fibKon}.
We begin by quickly recalling the definition of this measure and refer the reader to \cite[Sections 2 and 7]{spine} for a more detailed overview.

In \cite{Kon92}, Kontsevich defined a piecewise $2$-form $\omega_{\mathrm{Kon}}$ on $\MRG_{g,b}$ that computes intersection numbers on moduli space.
Restricting to a slice $\MRG_{g,b}(\bfL)$ with fixed boundary lengths, this form is seen to be symplectic on every maximal facet.
Thus, the Kontsevich form gives rise to volume forms
\[\frac{1}{(3g-3+b)!}
\bigwedge\nolimits^{3g-3+b}
\omega_{\mathrm{Kon}}\]
on each maximal facet, which can be glued together into a volume form on the entire slice $\MRG_{g,b}(\bfL)$. We will use $\eta_{\mathrm{Kon}}^{\bfL}$ to denote the measure associated to this volume form and refer to it as the {\em Kontsevich measure} on $\MRG_{g,b}(\bfL)$.
\footnote{See \cite[Remark 2.1]{spine} for a discussion of how to deal with this measure in the presence of non-trivial automorphism groups.}

When $\vg := (\vg_1,\dots,\vg_k)$ is an ordered, oriented simple closed multi-curve on $S_g$ with complementary subsurfaces $(\Sigma_j)_{j=1}^c$, 
we recall that we set
\[
\MRG(S_g \setminus \agamma; \bfL) :=
\prod_{j=1}^c \MRG_{g_j,b_j} \left(\bfL^{(j)}\right)
\]
for any length vector $\bfL \in \mathbb{R}_{>0}^k$.
As this is a product of moduli spaces with fixed boundary lengths, it has a (product) Kontsevich measure $\smash{\eta_\mathrm{Kon}^{\agamma, \bfL}}$.
Integrating against boundary lengths, the Kontsevich measures on each slice also fit together into a canonical measure on the total space, defined for any measurable subset $A \subseteq \MRG(S_g \setminus \agamma)$ by the formula
\[
\eta_{\mathrm{Kon}}^{\agamma}(A) := 
\int_{\mathbb{R}_{>0}^{k}} 
\eta_\mathrm{Kon}^{\agamma, \bfL}
\left( A \cap \MRG(S_g \setminus \agamma; \bfL) \right)
\,
dL_1 \ldots dL_k .
\]

We now use this to induce a measure on $\MRG(S_g \setminus \agamma;\Delta)$, the total space of the fibration over the standard simplex $\Delta \subset \mathbb{R}^k$. 
For $A \subseteq \mathcal{MRG}(S_g \setminus \agamma; \Delta)$ set
\[
\text{cone}(A) :=
\{ (\Gamma, t\mathbf{x}): (\Gamma, \mathbf{x}) \in A, t \in (0,1]\}
\subseteq \MRG(S_g \setminus \agamma).
\]
where $\Gamma$ is the underlying ribbon graph and $\mathbf{x}$ corresponds to its metric structure. 

Denote by $\rho_{g}(\vg) \in \mathbb{N}$ the number of components of $\vg$ that bound a torus with one boundary component.
Let $\sigma_{g}(\vg) > 0$ be the rational number given by
\[
\sigma_{g}(\vg) := \frac{\prod_{j=1}^c |K_{g_j,b_j}|}{|\text{Stab}_0(\vg)\cap K_{g}|},
\]
where $K_{g_j,b_j} \triangleleft \text{Mod}_{g_j,b_j}$ is the kernel of the mapping class group action on $\mathcal{T}_{g_j,b_j}$ and $K_{g} \triangleleft \text{Mod}_{g}$ is the kernel of the mapping class group action on $\mathcal{T}_{g}$.
These factors arise from special symmetries of moduli spaces of low-complexity surfaces; for a more extended discussion see \cite{Ara19b} and \cite{spine}.

\begin{definition}\label{def:fib_Kont}
The {\em fibered Kontsevich measure} $\mathring{\eta}_\mathrm{Kon}^{\Delta}$ on $\mathcal{MRG}(S_g \setminus \agamma; \Delta)$ is the measure which to every Borel measurable subset $A$ assigns the value
\[
\mathring{\eta}_\mathrm{Kon}^{\Delta}(A) := 
\frac{\sigma_{g}(\vg)}
{ 2^{\rho_{g}(\vg)}}
\int_{\text{cone}(A)} L_1 \cdots L_k \, d \eta_{\mathrm{Kon}}^{\agamma}(\Gamma, \mathbf{x}).
\]
Denote the total $\mathring{\eta}_\mathrm{Kon}^{\Delta}$-mass of $\mathcal{MRG}(S_g \setminus \agamma; \Delta)$ by $m_{\agamma}$.
\end{definition}

\subsection*{Integral points and horoball masses}
Let $\smash{m_{\agamma,h}^L}$ denote the total mass of $\smash{\hnu_{u,\agamma,h}^L}$ (equivalently, of $\smash{\tnu_{u,\agamma,h}^L}$).
Note that, since intersection numbers and square roots of extremal lengths scale homogeneously, equation \eqref{eqn:scaling} implies that
\begin{equation}\label{eqn:horomass_scale}
m_{\agamma,h}^L = L^{6g-6} m_{\agamma,h}^1.
\end{equation}

To compute $\smash{m_{\agamma,h}^1}$ and, in particular, to relate it to the (fibered) Kontsevich measure, we will need a better structural understanding of the unstable leaf $\smash{\Fol^u_{[\gamma]}}$. The following statement gives us the desired control; compare to the discussion of moderately slanted cylinder diagrams in \cite[Section 3]{Ara19a} and to the discussion of shear-shape coordinates for quadratic differentials in \cite{CF}.

\begin{lemma}\label{lem:unstable_bundle}
The critical graph map $\Xi \colon \mathcal{F}^u_{[\gamma]} / \Stab_0(\agamma) \to \MRG(S_g \setminus \vg)$ demonstrates the quotient $\Fol^u_{[\gamma]} / \Stab_0(\agamma)$ as a torus bundle over $\MRG(S_g \setminus \vg)$.
\end{lemma}

\begin{proof}
By definition, every $q \in \smash{\Fol^u_{[\gamma]}}$ is a unit-area quadratic differential whose horizontal foliation (i.e., imaginary part) is in the projective class of $[\gamma]$. In particular, this implies that its horizontal cylinders all have equal heights.
Cutting along the core curves of these cylinder, we are left with a flat cone structure with totally geodesic boundary on $S_g \setminus \agamma$.
Collapsing the vertical leaves then defines a deformation retract onto the critical graph $\Xi(q) \in \MRG(S_g \setminus \vg)$.

Conversely, given a tuple of metric ribbon graphs $(\Gamma,\bfx) \in \MRG(S_g \setminus \vg)$, there exists a unique choice of height
so that gluing together these metric ribbon graphs along cylinders of that height results in a unit area quadratic differential with the given horizontal separatrices; if $(\Gamma,\bfx) \in \MRG(S_g \setminus \vg; \bfL)$, then the corresponding height is $\smash{1/\sum_{i=1}^k L_i}$.
Compare to Figure \ref{fig:JSglue}.
The only ambiguity in this construction arises in choosing how much to shear along the cylinders of $\gamma$.
Thus, for each $\bfx \in \MRG(S_g \setminus \vg)$, there is a torus's worth of ways to construct a quadratic differential $q \in \mathcal{F}^u_{[\gamma]}$ with critical graph $\Xi(q) = (\Gamma,\bfx)$.
\end{proof}

\begin{figure}[ht]
    \centering
    \includegraphics[scale=.8]{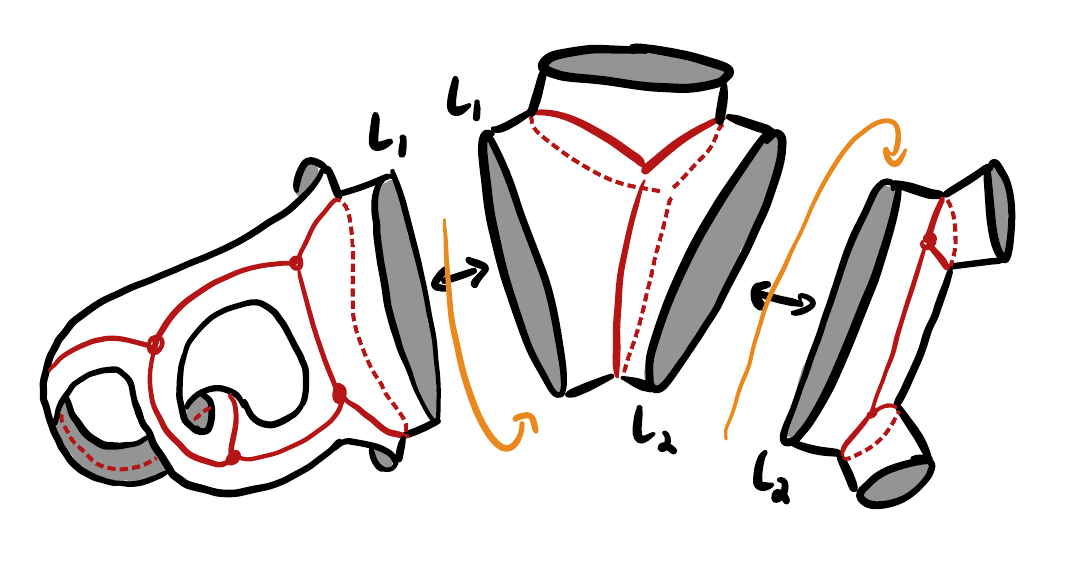}
    \caption{Thickening a ribbon graph and gluing boundaries to recover a Jenkins--Strebel differential with specified critical graph.}
    \label{fig:JSglue}
\end{figure}

The fibers of the critical graph map $\Xi \colon \mathcal{F}^u_{[\gamma]} / \Stab_0(\agamma) \to \MRG(S_g \setminus \vg)$ are tori of real dimension $k$ equal to the number of components of $\vg$. Each dimension represents twisting about one of the components.
Because of this, the fibers are naturally equipped with a notion of size coming from the circumferences of the corresponding cylinders (representing the possible amounts of twisting, plus a correction factor for extra symmetries). See also the discussion of the ``cut-and-glue fibration'' in \cite[Section 11]{spine}. 

This discussion allows us to express the pushforward of Thurston measure by $\Xi$ in terms of the Kontsevich measure on the base of the fibration.

\begin{proposition}\label{prop:pushTh_fibKon}
The following identity of measures on $\MRG(S_g \setminus \vg)$ holds:
\begin{equation}\label{eqn:pushTh_fibKon}
d \left( \Xi_* \widetilde{\nu}_{u,\vg} \right) = 
\frac{\sigma_g(\vg)}{2^{\rho_g(\vg)}}
L_1 \cdots L_k  \thinspace d\eta_{\mathrm{Kon}}^{\vg, \bfL} \, d L_1 \ldots dL_k.
\end{equation}
That is, for every measurable subset $A \subseteq \MRG(S_g \setminus \vg)$,
\[
\widetilde{\nu}_{u,\vg}(\Xi^{-1}A) = \frac{\sigma_g(\vg)}{2^{\rho_g(\vg)}} \int_{\mathbb{R}^k_{>0}} \eta_{\mathrm{Kon}}^{\vg, \bfL}(A \cap \mathcal{MRG}(S_g \setminus \vg;\mathbf{L})) \thinspace L_1\cdots L_k \thinspace dL_1\cdots dL_k.
\]

\end{proposition}
\begin{proof}
There are natural notions of integer points for both $\smash{\Fol^u_{[\gamma]} / \Stab_0(\vg)}$ and $\MRG(S \setminus \vg)$. In the first space, these are square-tiled surfaces. In the second space, these are integral metric ribbon graphs.
The measure $\mu_{\Th}$ can be defined as a weak-$\star$ limit of counting measures of integrally weighted simple closed multi-curves, and the corresponding measure $\nu_{u,\agamma}$ can hence be interpreted as a weak-$\star$ limit of counting measures of square-tiled surfaces.

Observe that the critical graph map $\Xi \colon \smash{\mathcal{F}^u_{[\gamma]} / \Stab_0(\agamma)} \to \MRG(S_g \setminus \vg)$ takes integer points to integer points.
Furthermore, over any integral metric ribbon graph in $(\Gamma,\bfx) \in \MRG(S \setminus \vg; \bfL)$, there are exactly $L_1 \cdots L_k$ square-tiled surfaces in $\Xi^{-1}(\Gamma,\bfx)$ (corresponding to integral amounts of twisting).
Thus, up to getting the correct normalization factor, it suffices to show that we can reinterpret the right-hand side of \eqref{eqn:pushTh_fibKon} in terms of counting integer points.
This statement was first observed by Norbury \cite{Norbury}, but is just a consequence of the fact that the Kontsevich measure is essentially Lebesgue measure in the lengths of edges. Compare with the discussion on the fibered Kontsevich measure in \cite{spine}.


To get the normalizing constant, we must be careful to count integer points weighted by the size of their automorphism group. Equivalently, we must ensure that the pushforwards of measures to orbifolds are weighted by their symmetries.
The symmetries of low-complexity moduli spaces (i.e., the kernel of the mapping class group action on these Teichm{\"u}ller spaces) hence contribute a factor of $\sigma_g(\vg)$.
The $2^{-\rho_g(\vg)}$ factor comes from the elliptic involution on $S_{1,1}$: its existence implies that there is only half as much twisting in the toral fibers as one might expect.
\end{proof}

Together with \eqref{eqn:horomass_scale} and the definition of the fibered Kontsevich measure (Definition \ref{def:fib_Kont}), Proposition \ref{prop:pushTh_fibKon} immediately implies the following:

\begin{corollary}\label{cor:horomass}
For any $L>0$ and any non-zero, continuous, compactly supported function
$h: \MRG(S_g \setminus \vg; \Delta) \to \RR$, the total mass $m_{\agamma,h}^L$ of $\widehat{\nu}_{u,\agamma,h}^L$ is equal to
\[m_{\agamma,h}^L = 
{L^{6g-6}} \thinspace
\int_{\mathcal{MRG}(S_g \setminus \agamma;\Delta)} h(\mrg) \thinspace d\mathring{\eta}_{\mathrm{Kon}}^{\Delta}(\mrg).\]
\end{corollary}

\subsection*{Equidistribution of critical-JS-horoballs}

We now show that the expanding pushforwards of critical-JS-horoballs equidistribute with respect to the Masur-Veech measure; this is the analogue of \cite[Theorem 5.2]{spine}.
As opposed to that paper, here we will be able to invoke strong results from Teichm{\"u}ller dynamics to deduce equidistribution relatively quickly; the approach presented here is one of several possible ones.
We first consider the horoball measures on unstable leaves.

\begin{proposition}\label{prop:unstable_horoball_equid}
For any non-zero, continuous, compactly supported function
$h: \MRG(S_g \setminus \vg; \Delta) \to \RR$, the following convergence holds with respect to the weak-$\star$ topology for measures on $\mathcal{Q}^1\mathcal{M}_g$:
    \[
    \lim_{L \to \infty} \frac{\widehat{\nu}_{u,\agamma,h}^L}{m_{\agamma,h}^L} = \frac{\widehat{\nu}_\MV}{b_g}.
    \]
\end{proposition}
\begin{proof}
This result follows directly from \cite[Theorem 1.6]{push}.
\footnote{An alternative proof of the desired equidistribution statement can be obtained using the mixing property of the Teichmüller geodesic flow; see for instance \cite[Proposition 3.2]{EMM19}. A proof using only the ergodicity of the Teichmüller horocycle flow should also follow from the methods discussed in \cite{Ara19b}.}
In fact, the cited theorem is stronger in the sense that it guarantees equidistribution of ``horospheres.'' The desired result for horoballs can be recovered by integrating horospheres along the direction of the Teichmüller geodesic flow.
\end{proof}

We now use Proposition \ref{prop:unstable_horoball_equid} to deduce the equidistribution of critical-JS-horoballs.
Unlike in the hyperbolic setting, see Theorem 5.2 and Corollary 5.3 in \cite{spine}, we cannot simply push Proposition \ref{prop:unstable_horoball_equid} down to $\M_g$ to arrive at the following result; this is related to the fact that $\nu_{\MV}$ is $\Mod_g$-invariant while $\nu_{u,\agamma}$ is not.

\begin{corollary}
\label{cor:MV_horoball_equid}
For any non-zero, continuous, compactly supported function
$h: \MRG(S_g \setminus \vg; \Delta) \to \RR$, the following convergence holds with respect to the weak-$\star$ topology for measures on $\mathcal{M}_g$:
\[
\lim_{L \to \infty} \frac{\widehat{\bfm}_{\agamma,h}^L}{m_{\agamma,h}^L} = \frac{\Lambda_g}{b_g} \thinspace \widehat{\bfm}.
\]
\end{corollary}

\begin{proof}
Let $f \colon \mathcal{M}_g \to \mathbb{R}$ be a continuous, compactly supported function and set $\smash{\widetilde{f}} \colon \T_g / \Stab_0(\agamma) \to \mathbb{R}$ to be its pullback to $\T_g / \Stab_0(\agamma)$. 
As a direct consequence of the definitions and the Hubbard--Masur theorem, for every $L > 0$ we can rewrite
\begin{align*}
\int_{\M_g} f(X) \, d \widehat{\bfm}^L_{\agamma, h}(X)
& = \int_{\T_g/\Stab_0(\agamma)} \smash{\widetilde{f}}(X) \, d \widetilde{\bfm}^L_{\agamma, h}(X)\\
& = \int_{\T_g/\Stab_0(\agamma)} \smash{\widetilde{f}}(X) \thinspace 
\lambda^+\left( \frac{\JS(X,\gamma)}{\sqrt{\Ext}_X(\gamma)} \right)
\, d (\pi_{[\gamma]})_* \widetilde{\nu}^L_{u,\agamma,h}(X) \\
& = \int_{\QoT/\Stab_0(\agamma)} \smash{\widetilde{f}}(\pi(q)) \thinspace \lambda^+(q) \, d \widetilde{\nu}^L_{u,\agamma,h}(q).
\end{align*}
Pushing back down to moduli space and dividing by the total mass $m_{\agamma,h}^L$ we get
\begin{align*}
\frac{1}{m_{\agamma,h}^L}\int_{\M_g} f(X) \, d {\widehat{\bfm}^L_{\agamma, h}}(X)
& = \frac{1}{m_{\agamma,h}^L} \int_{\QoM} f(\pi(q)) \thinspace \lambda^+(q) \, d {\widehat{\nu}^L_{u,\agamma, h}}(q) \\
& \to \frac{1}{b_g} \int_{\QoM} f({\pi}(q)) \thinspace \lambda^+(q) \, d \widehat{\nu}_{\mathrm{MV}}(q),
\end{align*}
where the convergence as $L \to \infty$ follows from Proposition \ref{prop:unstable_horoball_equid}. 
We can then integrate off the the $s_X$ factor using the product formula \eqref{eqn:MV_sphere_disint} and invoke Theorem \ref{thm:HMconst} to deduce the desired result.
\end{proof}

\section{Equidistribution of critical graphs}
\label{sec:avgunfold}

\subsection*{Outline of this section.} In this section we state and prove the main result of this paper in the case of a general multi-curve. Theorem \ref{theo:main_intro} follows as a special case.
The proof is obtained by putting the results of the previous sections into the outline discussed in the introduction.
Namely, after reducing the equidistribution problem at hand to a counting problem for curves whose Jenkins--Strebel differentials have constrained critical graphs, averaging and unfolding techniques allow us to further reduce to an equidistribution question for critical horoball measures.
Propositions \ref{prop:key_estimate} and \ref{prop:varA} will play an important role at this stage of the proof.
The results of \S\ref{sec:critical_equi}, which rely on the ergodicity of the Teichmüller horocycle flow, guarantee these measures equidistribute.
The relationship between the total mass of critical horoball measures and the Kontsevich measure, explained in Corollary \ref{cor:horomass}, then allows us to relate the asymptotics of our curve counting problem to the Kontsevich measure.

\subsection*{From equidistribution to counting.} 
Fix an ordered, oriented simple closed multi-curve $\vg$. 
For every $L > 0$, consider the extremal length counting function
\[
s(X,\agamma,L) := \# \{{\alpha} \in \mathrm{Mod}_g \cdot \agamma \ | \ \sqrt{\mathrm{Ext}}_X(\vec{\alpha}) \leq L\}.
\]
This does not depend on the marking of $X \in \mathcal{T}_g$ but only on its underlying conformal structure. Work of Mirzakhani \cite{Mir08b} gives sharp asymptotics for this count:
\begin{equation}
\label{eq:mirA}
\lim_{L \to \infty} \frac{s(X,\agamma,L)}{L^{6g-6}} = \frac{m_{\vg} \cdot \Lambda_g}{b_g},
\end{equation}
where $\Lambda_g > 0$ is the Hubbard--Masur constant, $b_g > 0$ is the total Masur--Veech volume of $\QoM$, and $m_{\vg}$ is the total mass of the fibered Kontsevich measure $\mathring{\eta}_{\mathrm{Kon}}^{\Delta}$ on the moduli space $\mathcal{MRG}(S_g \setminus \vg; \Delta)$.

\begin{remark}
As discussed in \cite[Remark 7.6]{spine}, the constant $m_{\vg}$ can be reconciled with the usual statement of \eqref{eq:mirA} involving the ``frequency'' $c(\gamma)$ by observing that we are counting ordered, oriented multi-curves (introducing a factor of $[\Stab(\gamma):\Stab_0(\agamma)]$) and that the coefficient of the top degree part of the Weil--Petersson volume polynomial of $S_g \setminus \agamma$ is exactly $m_{\agamma}$.
\end{remark}

Recall that our goal is to study the asymptotic distribution of the counting measures on $\mathcal{MRG}(S_g \setminus \vg; \Delta)$ given by
\[
\eta_{X,\agamma}^L := \sum_{\vec{\alpha} \in \mathrm{Mod}_g \cdot \agamma} \mathbbm{1}_{[0,L]}\left(\sqrt{\Ext}_X({\alpha})\right) \cdot \delta_{\Sp^1(X, \vec{\alpha})}.
\]
The following is the general version of Theorem \ref{theo:main_intro}; its proof will occupy the rest of this section. Compare with \cite[Theorem 7.7]{spine}.

\begin{theorem}
\label{theo:main_2A}
Let $\agamma:=(\agamma_1,\dots,\agamma_k)$ be an ordered, oriented simple closed multi-curve on $S_g$ and $X \in \mathcal{M}_g$ be a complex structure on $S_g$. Then
\[
\lim_{L \to \infty} \frac{\eta_{X,\agamma}^L}{s(X,\agamma,L)} = \frac{\mathring{\eta}_{\mathrm{Kon}}^{\Delta}}{m_{\agamma}}
\]
with respect to the weak-$\star$ topology for measures on $\mathcal{MRG}(S_g \setminus \agamma;\Delta)$.
\end{theorem}

\begin{remark}
For the sake of brevity and readability, throughout the rest of this section we will shorten $\MRG(S_g \setminus \agamma;\Delta)$ to $\MRG$. We will not consider any other spaces of ribbon graphs in the sequel.
\end{remark}

As explained in \S\ref{sec:intro}, Theorem \ref{theo:main_2A} is equivalent to a counting problem for metric ribbon graphs. 
More concretely, it is enough to show that for every continuous, compactly supported $f \colon \mathcal{MRG} \to \mathbb{R}_{\geq 0}$,
\begin{gather*}
\lim_{L \to \infty} \frac{1}{s(X,\agamma,L)} \int_{\mathcal{MRG}} f(\bfx) \thinspace d\eta_{X,\agamma}^L(\bfx) 
= \frac{1}{m_{\agamma}}\int_{\mathcal{MRG}} f(\bfx) \thinspace d\mathring{\eta}_{\mathrm{Kon}}^{\Delta}(\bfx).
\end{gather*}
For every $L > 0$ consider the $f$-weighted counting function
\begin{align}
\label{eq:countA}
c(X,\agamma,f,L) &:= 
\int_{\mathcal{MRG}} f(\bfx) \thinspace d\eta_{X,\agamma}^L(\bfx)\\
&\phantom{:}= 
\sum_{\vec{\alpha} \in \mathrm{Mod}_g \cdot \agamma} \mathbbm{1}_{[0,L]}(\sqrt{\Ext}_X({\alpha})) \cdot f(\Sp^1(X, \vec{\alpha})). \nonumber
\end{align}
Observe that taking $f \equiv 1$, we recover the usual counting function $s(X, \agamma, L)$.

The rest of this section is devoted to proving the ensuing result, from which Theorem \ref{theo:main_2A} follows directly by the above discussion. This is a generalized version of Theorem \ref{theo:red_introA} from the introduction in the case of a general multi-curve.

\begin{theorem}
\label{theo:redA}
Let $\agamma:=(\agamma_1,\dots,\agamma_k)$ be an ordered, oriented simple closed multi-curve on $S_g$ and $X \in \mathcal{M}_g$ be a complex structure on $S_g$.
Then, for every continuous, compactly supported $f \colon \mathcal{MRG} \to \mathbb{R}_{\geq 0}$,
\[
\lim_{L \to \infty} \frac{c(X,\agamma,f,L)}{s(X,\agamma,L)} 
= \frac{1}{m_{\agamma}}\int_{\mathcal{MRG}} f(\bfx) \thinspace d\mathring{\eta}_{\mathrm{Kon}}^{\Delta}(\bfx).
\]
\end{theorem}

For the rest of this section we fix a marked conformal structure $X \in \mathcal T_g$, an ordered, oriented simple closed multi-curve $\vg$ on $S_g$, and a non-zero, non-negative, continuous, compactly supported function $f\colon \mathcal{MRG} \to \mathbb{R}_{\geq 0}.$

\subsection*{Averaging counts.}
Our next goal is to average the counting functions introduced in (\ref{eq:countA}) over small neighborhoods of moduli space. Using the results from \S \ref{sec:weightsvaryA}, we first study how these counting functions vary in such neighborhoods. 

Recall that Proposition \ref{prop:varA} states that for every pair of conformal structures $X, Y \in \T_g$ that are sufficiently close in the Teichmüller metric and most ordered, oriented simple closed multi-curves $\vec{\alpha}$ on $S_g$, the critical graphs $\Sp^1(X,\vec{\alpha})$ and $\Sp^1(Y,\vec{\alpha})$ belong to the same facet of $\mathcal{MRG}$ and the corresponding edges have length differing by a multiplicative constant.
We now define analogous neighborhoods in the moduli space of ribbon graphs.

Given $\mrg \in \mathcal{MRG}$ and a positive constant $\varepsilon> 0$, denote by $N_{\varepsilon}(\mrg)$ the set of all $\mrgy \in \mathcal{MRG}$ in the same facet as $\mrg$, i.e., with the same topological type of underlying ribbon graph as $\mrg$, such that for every edge $e$ of $\mrg$ and $\mrgy$,
\[
e^{-\varepsilon} \cdot |e|_\mrg \leq |e|_\mrgy \leq e^\varepsilon \cdot |e|_\mrg;
\]
here we have implicitly fixed a local marking so we can compare the weights of specific edges.
For every $\varepsilon > 0$ consider the averaged functions
\[
f_{\varepsilon}^{\min}, f_{\varepsilon}^{\max} \colon \mathcal{MRG} \to \mathbb{R}_{\geq 0}
\]
given by
\[
f_{\varepsilon}^{\min}(\mrg) :=\min_{\mrgy \in N_{\varepsilon}(\mrg)} f(\mrgy),\quad
f_{\varepsilon}^{\max}(\mrg) := \max_{\mrgy \in N_{\varepsilon}(\mrg)} f(\mrgy).
\]

We begin our proof of Theorem \ref{theo:redA} with the following estimate, which allows us to compare counts of curves on $X$ with counts on nearby surfaces.
Propositions \ref{prop:key_estimate} and \ref{prop:varA} play a crucial role in the proof of this result. For $X \in \M_g$, a function $f$ on $\MRG$, and functions $F, G$ on $\RR$, we say that $F = O_{X,f}(G)$ if $F = O(G)$ where the implicit constants depend only on $X$ and $f$.
 
\begin{proposition}
\label{prop:compA}
For every $\varepsilon > 0$ there exists $\delta := \delta(X,\varepsilon) \in (0,\varepsilon)$ such that for every $Y \in \mathcal{M}_g$ with $d_\mathrm{Teich}(X,Y) < \delta$, the following estimates hold:
\begin{gather}
    c\left(Y,\agamma,f_{\varepsilon}^{\min},e^{-\delta} L\right) + O_{X,f}\left(L^{6g-7} + \varepsilon \cdot L^{6g-6}\right) \leq c(X,\agamma,f,L), \label{eq:bd_1A}\\
    c(X,\agamma,f,L) \leq c\left(Y,\agamma,f_{\varepsilon}^{\max},e^\delta L\right) + O_{X,f}\left(L^{6g-7} + \varepsilon \cdot L^{6g-6}\right). \label{eq:bd_2A}
\end{gather}
\end{proposition}
\begin{proof}
We prove \eqref{eq:bd_2A}. Similar arguments yield a proof of (\ref{eq:bd_1A}). 

We begin by applying Proposition \ref{prop:key_estimate} to  focus our attention on curves in the mapping class group orbit of $\agamma$ whose corresponding Jenkins-Strebel differentials have no short saddle connections.
Recall that for any compact set $\mathcal K \subseteq \T_g$ and any $\sigma > 0$, we use $F(\mathcal K, \sigma) \subseteq \mathcal{MF}_g(\mathbb{Z})$ to denote the set of integrally weighted simple closed multi-curves $\alpha$ on $S_g$ such that $\JS(Y, \alpha)$, after rescaling to have unit area, has a $\sigma$-short saddle connection for some $Y \in \mathcal K$. By abuse of notation we also use $F(\mathcal K, \sigma)$ to denote the set of ordered, oriented simple closed multi-curves $\vec{\alpha}$ on $S_g$ whose underlying multi-curve belongs to $F(\mathcal K, \sigma)$. 
Now take $\mathcal{K} \subseteq \mathcal{T}_g$ to be the closed unit ball in the Teichmüller metric centered at $X$ and $\sigma = \varepsilon > 0$ arbitrary.
Consider the truncated counting function
\begin{gather*}
c^\dagger(X,\agamma,f,L) 
:= \sum_{\vec{\alpha} \in \mathrm{Mod}_g \cdot \agamma \setminus F(\mathcal{K},\varepsilon)}
\mathbbm{1}_{[0, L]}\left(\sqrt{\mathrm{Ext}}_X(\alpha)\right) \cdot f(\Sp^1(X,\vec{\alpha})).
\end{gather*}
By Proposition \ref{prop:key_estimate} it follows that
\begin{equation}
\label{eq:a_2A}
c(X,\agamma,f,L) = c^\dagger(X,\agamma,f,L) + \|f\|_\infty \cdot O_\mathcal{K}\left(L^{6g-7} + \varepsilon \cdot L^{6g-6}\right).
\end{equation}

We can now invoke the geometric comparison results of Section \ref{sec:weightsvaryA}. Consider $\delta = \delta(\mathcal{K},\varepsilon,\varepsilon) > 0$ as in Proposition \ref{prop:varA} and set 
\[\delta' := \min \{1,\delta,\varepsilon\}.\] 
Now for any $Y \in \mathcal{T}_g$ such that $d_\mathrm{Teich}(X,Y) < \delta'$ and any  $\vec{\alpha} \in \mathrm{Mod}_g \cdot \agamma \setminus F(\mathcal{K},\varepsilon),$
it follows from Proposition \ref{prop:varA} that $\Sp^1(X,\vec{\alpha})$ and $\Sp^1(Y,\vec{\alpha})$ are in the same facet of $\mathcal{MRG}$ and that for every edge $e$ of such metric ribbon graphs,
\begin{equation*}
e^{-\varepsilon} \cdot |e|_{\Sp^1(Y,\vec{\alpha})}  \leq |e|_{\Sp^1(X,\vec{\alpha})} \leq e^\varepsilon \cdot |e|_{\Sp^1(Y,\vec{\alpha})}.
\end{equation*}
It follows that $\Sp^1(X,\vec{\alpha}) \in N_{\varepsilon}(\Sp^1(Y,\vec{\alpha}))$, so by definition
\begin{equation}
\label{eq:d1}
f(\Sp^1(X,\vec{\alpha})) \leq f^{\max}_{\varepsilon}(\Sp^1(Y,\vec{\alpha})).
\end{equation}

By Kerckhoff's characterization of the Teichmüller metric (Theorem \ref{thm:Kerckhoff}),
\begin{equation}
\label{eq:d2}
\sqrt{\Ext}_Y(\alpha) \leq e^{\delta'} \cdot\sqrt{\Ext}_X(\alpha).
\end{equation}
From \eqref{eq:d1} and \eqref{eq:d2} we deduce
\begin{equation}
\label{eq:a_3A}
c^\dagger(X,\agamma,f,L) \leq c(Y,\agamma,f^{\max}_{\varepsilon},e^{\delta'} L).
\end{equation}
Putting together (\ref{eq:a_2A})  and (\ref{eq:a_3A}) we conclude  
\[
c(X,\agamma,f,L) \leq c\left(Y,\agamma,f_{\varepsilon}^{\max},e^{\delta'} L\right) + \|f\|_\infty \cdot O_{\mathcal{K}}\left(L^{6g-7} + \varepsilon \cdot L^{6g-6}\right). \qedhere
\]
\end{proof}


We now integrate Proposition \ref{prop:compA} against the pushforward $\hbfm$ of Masur--Veech measure.
For every $\delta \in (0,1)$ denote by $U_X(\delta) \subseteq \mathcal{M}_g$ the open ball of radius $\delta$ centered at $X \in \mathcal{M}_g$ with respect to the Teichmüller metric and let $\beta_{X,\delta} \colon \mathcal{M}_g \to \mathbb{R}_{\geq 0}$ be any bump function supported on $U_X(\delta)$ of total $\mathbf{m}$ mass $1$. 

\begin{corollary}\label{cor:comp_2A}
Let all notation be as in Proposition \ref{prop:compA}. Then $c(X, \vg, f, L)$ is bounded below by
\begin{equation} \label{eq:bd_4A}
\int_{\mathcal{M}_g} \beta_{X,\delta}(Y) \cdot c\left(Y,\agamma,f_{\varepsilon}^{\min},e^{-\delta} L\right) \thinspace d\widehat{\mathbf{m}}(Y) + O_{X,f}\left(L^{6g-7} + \varepsilon L^{6g-6}\right) 
\end{equation}
and above by 
\begin{equation}
\label{eq:bd_5A}
\int_{\mathcal{M}_g} \beta_{X,\delta}(Y) \cdot c\left(Y,\agamma,f_{\varepsilon}^{\max},e^\delta L\right) \thinspace d\widehat{\mathbf{m}}(Y) + O_{X,f}\left(L^{6g-7} + \varepsilon L^{6g-6}\right). 
\end{equation}
\end{corollary}

\subsection*{Unfolding averaged counts.}
Unfolding the integrals in \eqref{eq:bd_4A} and \eqref{eq:bd_5A} over $\mathcal{T}_{g} / \text{Stab}_0(\agamma)$ and pushing them back down to $\mathcal{M}_{g}$ in a suitable way will reduce the proof of Theorem \ref{theo:redA} to an applicaton of Corollary \ref{cor:MV_horoball_equid}.
The following proposition describes this principle; 
the reader should also compare to \cite[Proposition 3.3]{Ara20a} and \cite[Proposition 6.6]{spine}.

\begin{proposition}
	\label{prop:pull_pushA}
	Fix a continuous, compactly supported function $h \colon \mathcal{MRG} \to \mathbb{R}_{\geq 0}$. Then for every $\delta > 0$ and every $L > 0$,
	\begin{equation}
	\label{eq:unfoldingA}
	\int_{\mathcal{M}_{g}} \beta_{X,\delta}(Y) \cdot c(Y,\agamma,h,L) \thinspace d\widehat{\mathbf{m}}(Y) = \int_{\mathcal{M}_{g}} \beta_{X,\delta}(Y) \thinspace d\hbfm_{\agamma,h}^{L}(Y).
	\end{equation}
\end{proposition}

\begin{remark}
Notice that our weight function has changed names; this is because we eventually apply Proposition \ref{prop:pull_pushA} with $h$ equal to the functions $f_{\varepsilon}^{\max}$ and $f_{\varepsilon}^{\min}$.
\end{remark} 

\begin{proof}
Let $\delta > 0$ and $L > 0$ be arbitrary. For every $Y \in \mathcal{M}_{g}$ one can rewrite the counting function $c(Y,\agamma,h,L)$ as follows:
	\begin{align*}
	c(Y,\agamma,h,L)  &= \sum_{\vec{\alpha} \in \text{Mod}_{g} \cdot \agamma} \mathbbm{1}_{[0,L]}\left(\sqrt{\Ext}_Y(\alpha)\right) \cdot h(\Sp^1(Y,\vec{\alpha}))\\
	&= \sum_{[\phi] \in \text{Mod}_{g}/\text{Stab}_0(\agamma)} \mathbbm{1}_{[0,L]}\left(\sqrt{\Ext}_Y(\phi. \agamma)\right) \cdot h(\Sp^1(Y,\phi.\agamma)) \\
	&= \sum_{[\phi] \in \text{Mod}_{g}/\text{Stab}_0(\agamma)} \mathbbm{1}_{[0,L]}\left(\sqrt{\Ext}_{\phi^{-1}.Y}(\agamma)\right) \cdot h(\Sp^1(\phi^{-1}.Y,\agamma)) \\
	&= \sum_{[\phi] \in \text{Stab}_0(\agamma) \backslash \text{Mod}_{g}} \mathbbm{1}_{[0,L]}\left(\sqrt{\Ext}_{\phi.Y}(\agamma)\right) \cdot h(\Sp^1(\phi.Y,\agamma)). \\
	\end{align*}
	Let us record this fact as
	\begin{equation}
	\label{eq:count_masA}
	c(Y,\agamma,h,L)  = \sum_{[\phi] \in \text{Stab}_0(\agamma) \backslash \text{Mod}_{g}} \mathbbm{1}_{[0,L]}\left(\sqrt{\Ext}_{\phi.Y}(\agamma)\right) \cdot h(\Sp^1(\phi.Y,\agamma)).
	\end{equation}

Denote by
$p_{\agamma}: \mathcal{T}_{g}/\text{Stab}_0(\agamma) \to \mathcal{M}_{g}$
the quotient map and let $\smash{\widetilde{\beta}_{X,\delta}} = \beta_{X,\delta} \circ p_{\agamma}$ be the lift of $ \beta_{X,\delta}$ to this cover. Unfolding the integral on the left hand side of (\ref{eq:unfoldingA}) using (\ref{eq:count_masA}) it follows that
	\begin{align*}
	\int_{\mathcal{M}_{g}} \beta_{X,\delta}(Y)& \cdot c (Y,\agamma,h,L) \thinspace d \widehat{\mathbf{m}}(Y) \\
	& = \int_{\mathcal{T}_{g}/\text{Stab}_0(\agamma)} \widetilde{\beta}_{X,\delta}(Y) \cdot \mathbbm{1}_{[0,L]}\left(\sqrt{\Ext}_Y(\agamma)\right) \cdot h(\Sp^1(Y,\agamma)) \thinspace
	d \tbfm(Y)  \\
	& = \int_{\mathcal{T}_{g}/\text{Stab}_0(\agamma)} \widetilde{\beta}_{X,\delta}(Y)
	\thinspace d \tbfm_{\agamma,h}^{L}(Y) \\
	& = \int_{\mathcal{M}_{g}} \beta_{X,\delta}(Y)
	\thinspace d \hbfm_{\agamma,h}^{L}(Y),
	\end{align*}
where the second equality follows from the definition of the horoball measure $\tbfm_{\agamma,h}^{L}$ appearing in \eqref{eq:Xihoro_meas} and the third equality follows by taking the pushforward.
\end{proof}

\subsection*{Reducing counting to equidistribution.} 
We can now finish the proof of Theorem \ref{theo:redA} by applying our equidistribution results from \S \ref{sec:critical_equi}.
Proposition \ref{prop:pull_pushA} relates averages of the $f$-weighted counting function $c(X,\agamma,f,L)$ to horoball measures; our strategy now is to relate the original counting function and the mass of these measures, which we can then compare with the count $s(X,\agamma,L)$ of all ordered, oriented simple closed multi-curves in the $\Mod_g$-orbit of $\agamma$.

\begin{proof}[Proof of Theorem \ref{theo:redA}]
Recall that we are aiming to prove that
\[
\lim_{L \to \infty} \frac{c(X,\agamma,f,L)}{s(X,\agamma,L)} 
= \frac{1}{m_{\agamma}}\int_{\mathcal{MRG}} f(\mrg) \thinspace d\mathring{\eta}_{\mathrm{Kon}}^{\Delta}(\mrg).
\]

By Corollary \ref{cor:horomass}, we know that for any continuous, compactly supported function $h \colon \mathcal{MRG} \to \mathbb{R}_{\geq 0}$, the total mass $m_{\vg, h}^L$ of the unstable horoball measure $\widehat{\mathbf{m}}_{\vg, h}^L$ on $\mathcal{M}_g$ is $L^{6g-6}$ times the integral
\[r(\vg, h) := \int_{\MRG} h(\bfx) \thinspace d\mathring{\eta}_{\mathrm{Kon}}^{\Delta}(\mrg).\]
Proving Theorem \ref{theo:redA} is then equivalent to showing that both of the following inequalities hold:
\begin{gather}\label{eq:count_lbdA}
	\frac{r(\agamma, f)}{m_{\agamma}} \leq \liminf_{L \to \infty} \frac{c(X,\agamma,f,L)}{s(X,\agamma,L)},\\
	\label{eq:count_ubdA}
	\limsup_{L \to \infty} \frac{c(X,\agamma,f,L)}{s(X,\agamma,L)} \leq \frac{r(\agamma,f)}{m_{\agamma}}.
\end{gather}
	
We verify (\ref{eq:count_ubdA}) by averaging and unfolding; a proof of (\ref{eq:count_lbdA}) can be obtained following the same argument.
Let $\varepsilon \in (0,1)$ be arbitrary and $\delta = \delta(X,\varepsilon)$ as in Corollary \ref{cor:comp_2A}.
Shrinking $\delta$ as necessary, we can also assume that $e^{(6g-6){\delta}} \le 2$.
Set $h := f_{\varepsilon}^{\max}$; Corollary \ref{cor:comp_2A} then implies we can average our counting function to get
	\[
	c(X,\agamma,f,L) \leq \int_{\mathcal{M}_g} \beta_{X,\delta}(Y) \cdot c\left(Y,\agamma,h,e^\delta L\right) \thinspace d\widehat{\mathbf{m}}(Y) + O_{X,f}\left(L^{6g-7} + \varepsilon L^{6g-6}\right).
	\]
Set $L' := e^{\delta} L$. Unfolding the integral, i.e., using Proposition \ref{prop:pull_pushA}, we deduce
	\[
	c(X,\agamma,f,L) \leq \int_{\mathcal{M}_{g}} \beta_{X,\delta}(Y) \thinspace d\hbfm_{\agamma,h}^{L'}(Y) + O_{X,f}\left(L^{6g-7} + \varepsilon L^{6g-6}\right).
	\]
	Dividing this inequality by $m_{\agamma,h}^{L'}$ (which is nonzero so long as $f \neq 0$) we get
\[
\frac{c(X,\agamma,f,L)}{m_{\agamma,h}^{L'}} \leq 
\int_{\mathcal{M}_{g}} \beta_{X,\delta}(Y) \thinspace 
 d\frac{\hbfm_{\agamma,h}^{L'}}{m_{\agamma,h}^{L'}}(Y) + 
\frac{O_{X,f}(\varepsilon)}{r(\vg, h)} ,
\]
where we have invoked Corollary \ref{cor:horomass} and our assumption on $\delta$ to simplify the bound on the far right.
Since $h \ge f$, we know that $r(\agamma, h) \ge r(\agamma, f)$, so we can also absorb this term into our big $O$ estimate.

Taking the $\limsup$ as $L \to \infty$ and applying Corollary \ref{cor:MV_horoball_equid},
we deduce that
\begin{equation}\label{eq:fcount_vs_RSCmassA}
\limsup_{L \to \infty} \frac{c(X,\agamma,f,L)}{m_{\agamma,h}^{L'}} \leq \frac{\Lambda_g}{b_g}\int_{\mathcal{M}_{g}} \beta_{X,\delta}(Y) \thinspace d\widehat{\mathbf{m}}(Y) + O_{X,f}(\varepsilon)
= \frac{\Lambda_g}{b_g} + O_{X,f}(\varepsilon).
\end{equation}
Combining this with Mirzakhani's asymptotic count \eqref{eq:mirA} and our expression for horoball masses (Corollary \ref{cor:horomass}),
we arrive at the following estimate:
\begin{equation}\label{eqn:cs_bd_const}
\limsup_{L \to \infty} 
\frac{c(X,\agamma,f,L)}{s(X,\agamma,L)} 
\leq \frac{r(\agamma,h) \cdot e^{(6g-6)\delta}}{m_{\agamma}} + O_{X,f}(\varepsilon).
\end{equation}

We now shrink our approximating neighborhoods.
By definition, $h :=f_{\varepsilon}^{\max} \searrow f$ pointwise as $\varepsilon \searrow 0$. In particular, by the monotone convergence theorem,
\begin{align*}
	\lim_{\varepsilon \searrow 0}  r(\agamma,f_{\varepsilon}^{\max})
 = \lim_{\varepsilon \searrow 0}  \int_{\MRG} f_{\varepsilon}^{\max}(\mrg) \thinspace d\mathring{\eta}_{\mathrm{Kon}}^{\Delta}(\mrg) 
 =   \int_{\MRG} f(\mrg) \thinspace d\mathring{\eta}_{\mathrm{Kon}}^{\Delta}(\mrg) 
	&= r(\agamma,f).
	\end{align*}
Sending $\varepsilon$, hence $\delta$, to $0$ we get that the right-hand side of \eqref{eqn:cs_bd_const} converges to $r(\vg, f)/ m_{\agamma}$ as desired.
\end{proof}

This completes the proof of the counting result (Theorem \ref{theo:redA}), hence the proof of our main equidistribution result (Theorem \ref{theo:main_2A}).

\subsection*{Simultaneous equidistribution.} Drawing inspiration from \cite{AES16a,AES16b,ERW19, spine}, we now discuss the issue of simultaneous equidistribution. More concretely, we show that the placement of simple closed multi-curves in the space of singular measured foliations is asymptotically independent from the critical graph of the Jenkins--Strebel differential they define on a given Riemann surface.

Recall that $\mathcal{PMF}_g$ denotes the space of projective singular measured foliations on $S_g$ and that $[\lambda] \in \mathcal{PML}_g$ denotes the projective class of $\lambda \in \mathcal{MF}_g$. Given $X \in \mathcal{T}_g$, consider the coned-off Thurston measure $\mu_\mathrm{Thu}^X$ which to every measurable subset $A \subseteq \mathcal{PML}_g$ assigns the value
\[
\mu_\mathrm{Thu}^X(A) := \mu_\mathrm{Thu}(\{\lambda \in \mathcal{ML}_g \ | \ \mathrm{Ext}_X(\lambda) \leq 1, \ [\lambda] \in A\}).
\]
By Lemmma \ref{lem:HMfunc_Ext} and Theorem \ref{thm:HMconst}, the total mass of this measure is precisely $\Lambda(X) = \Lambda_g > 0$, the Hubbard--Masur constant. Furthermore, as discussed in \cite[Proposition 2.3]{ABEM12}, this measure can be related to the fiberwise measures $s_X$ and the Hubbard--Masur function $\lambda^+$ as follows.

\begin{proposition}
    \label{prop:ABEMb}
    Let $X \in \mathcal{T}_g$ be a marked complex structure on $S_g$. Consider the homeomorphism $[\Im] \colon S(X) \to \mathcal{PMF}_g$. Then,
    \[
    \lambda^+(q) = \frac{d [\Im]_* \thinspace \mu_{\mathrm{Thu}}^X}{ds_X}(q).
    \]
\end{proposition}

Fix an ordered, oriented simple closed multi-curve $\vg := (\vg_1,\dots,\vg_k)$ on $S_g$ and a marked hyperbolic structure $X \in \mathcal{T}_g$. Recall that $\gamma \in \mathcal{MF}_g$ denotes the equivalence class of $\vg$ as a singular measured foliation on $S_g$. For every $L > 0$ consider the counting measure on $\mathcal{PML}_g$ given by
\[
\zeta_{\agamma,X}^L := \sum_{\vec{\alpha} \in \mathrm{Mod}_g \cdot \vg} \mathbbm{1}_{[0,L]}\left(\sqrt{\mathrm{Ext}}_X(\alpha)\right) \cdot \delta_{[\alpha]}.
\]
Observe that we weight a given projective class $[\alpha]$ by the index $[\Stab(\gamma):\Stab_0(\agamma)]$; this allows us to use the oriented count $s(X, \vg, L)$ below.
The following result can be deduced directly from Mirzakhani's work \cite[Theorem 6.4]{Mir08b}.

\begin{theorem}
\label{theo:mir}
In the weak-$\star$ topology for measures on $\mathcal{PML}_g$,
\[
\lim_{L \to \infty} \frac{\zeta_{\agamma,X}^L}{s(X,\vg,L)} = \frac{\mu_\mathrm{Thu}^X}{\Lambda_g}.
\]
\end{theorem}

It is natural to consider the question of simultaneous equidistribution for the limits in Theorems \ref{theo:main_2A} and \ref{theo:mir}. More precisely, fix an ordered, oriented simple closed multi-curve $\vg := (\vg_1,\dots,\vg_k)$ on $S_g$ and $X \in \mathcal{T}_g$. For every $L > 0$ consider the counting measure on $\mathcal{PMF}_g \times \mathcal{MRG}(S_g \setminus \agamma;\Delta)$ given by
\[
\xi_{\vg,X}^L := \sum_{{\vec{\alpha}} \in \mathrm{Mod}_g \cdot {\vg}} \mathbbm{1}_{[0,L]}\left(\sqrt{\mathrm{Ext}}_X({\alpha})\right) \cdot \delta_{[\alpha]} \otimes \delta_{\Xi^1(X,{\vec{\alpha}})}.
\]
As always, this measure depends only on the underlying hyperbolic structure of $X \in \mathcal{T}_g$ and not on its marking. The question of equidistribution as $L \to \infty$ of these measures can be tackled by the same methods used in the proof of Theorem \ref{theo:main_2A} subject to some important modifications we now discuss.

The most important difference in the proof comes from the equidistribution result that one must use instead of Corollary \ref{cor:MV_horoball_equid}. Denote by $\mathcal{P}^1\mathcal{T}_g = \mathcal{T}_g \times \mathcal{PMF}_g$ the bundle projective singular measured foliations over Teichmüller space. This bundle carries a natural measure $\nu$ given by the following disintegration formula
\[
d\mathbf{n}(X,[\lambda]) = d\mu_{\mathrm{Thu}}^X([\lambda]) \thinspace d\mathbf{m}(X).
\]
The quotient $\mathcal{P}^1\mathcal{M}_g = \mathcal{P}^1\mathcal{T}_g/\mathrm{Mod}_g$ by the diagonal action of the mapping class group is the bundle of projective singular measured foliations over moduli space. As this action preserves the measure $\mathbf{n}$, we obtain a measure $\widehat{\mathbf{n}}$ on $\mathcal{P}^1\mathcal{M}_g$ satisfying the following disintegration formula:
\[
d\widehat{\mathbf{n}}(X,[\lambda]) = d\mu_{\mathrm{Thu}}^X([\lambda]) \thinspace d\widehat{\mathbf{m}}(X).
\]

To define the relevant horoballs we want to consider over $\mathcal{P}^1\mathcal{M}_g$ we proceed as follows. For any $L>0$ and any non-zero, continuous, compactly supported function
$h: \MRG(S_g \setminus \vg; \Delta) \to \RR$, define the $\Xi^1$-horoball measure on $\mathcal{P}^1\T_g$ by
\begin{equation*}
d\mathbf{n}_{\agamma,h}^L(X,[\lambda]) := d \delta_{[\gamma]}([\lambda]) \thinspace d\bfm_{\vg,h}^L(X),
\end{equation*}
where $\delta_{[\gamma]}$ denotes the delta mass at $[\gamma] \in \mathcal{PMF}_g$. Directly from the definitions one can check that this measure is $\mathrm{Stab}_0(\vg)$-invariant. It follows that one can locally push this measure forward to $\mathcal{P}^1\mathcal{T}_g/\mathrm{Stab}_0(\vg)$ to get a measure $\smash{\widetilde{\mathbf{n}}_{\agamma,h}^L}$. Denote by $\smash{\widehat{\mathbf{n}}_{\agamma,h}^L}$ the measure on $\mathcal{P}^1\mathcal{M}_g$ obtained by pushing forward this measure. Notice that the total mass of this measure is exactly $m_{\vg,h}^L>0$.

The following result is the main equidistribution result needed to address the simultaneous equidistribution question alluded to above. 

\begin{proposition}
    \label{prop:XX}
    For any non-zero, continuous, compactly supported function
    $h: \MRG(S_g \setminus \vg; \Delta) \to \RR$, the following convergence holds with respect to the weak-$\star$ topology for measures on $\mathcal{P}^1\mathcal{M}_g$:
    \[
    \lim_{L \to \infty} \frac{\widehat{\mathbf{n}}_{\agamma,h}^L}{m_{\agamma,h}^L} = \frac{\widehat{\mathbf{n}}}{b_g} \thinspace.
    \]
\end{proposition}

\begin{proof}
    Let $f \colon \mathcal{P}^1\mathcal{M}_g \to \mathbb{R}$ be a continuous, compactly supported function and set $\smash{\widetilde{f}} \colon \mathcal{P}^1\T_g / \Stab_0(\agamma) \to \mathbb{R}$ to be its pullback. 
As a direct consequence of the definitions and the Hubbard--Masur theorem, for every $L > 0$ we can rewrite
\begin{align*}
\int_{\mathcal{P}^1\M_g} f(X,[\lambda]) \, &d \widehat{\mathbf{n}}^L_{\vg, h}(X,[\lambda])\\
& = \int_{\mathcal{P}^1\T_g/\Stab_0(\vg)} \smash{\widetilde{f}}(X,[\lambda]) \, d \widetilde{\mathbf{n}}^L_{\agamma, h}(X,[\lambda])\\
& = \int_{\mathcal{T}_g/\Stab_0(\vg)} \smash{\widetilde{f}}(X,[\gamma]) \thinspace 
\lambda^+\left( \frac{\JS(X,\gamma)}{\sqrt{\Ext}_X(\gamma)} \right)
\, d (\pi_{[\gamma]})_* \widetilde{\nu}^L_{u,\agamma,h}(X) \\
& = \int_{\QoT/\Stab_0(\vg)} \smash{\widetilde{f}}(\pi(q),[\Im(q)]) \thinspace \lambda^+(q) \, d \widetilde{\nu}^L_{u,\agamma,h}(q).
\end{align*}
Pushing back down to moduli space and dividing by the total mass $m_{\agamma,h}^L$ we get
\begin{align*}
\frac{1}{m_{\agamma,h}^L} \int_{\mathcal{P}^1\M_g} f(X,[\lambda]) \, &d \widehat{\mathbf{n}}^L_{\vg, h}(X,[\lambda]) \\
& = \frac{1}{m_{\agamma,h}^L} \int_{\QoM} f(\pi(q),[\Im(q)]) \thinspace \lambda^+(q) \, d \widehat{\nu}^L_{u,\agamma,h}(q) \\
& \to \frac{1}{b_g} \int_{\QoM} f(\pi(q),[\Im(q)])  \thinspace \lambda^+(q) \, d \widehat{\nu}_{\mathrm{MV}}(q),
\end{align*}
where the convergence as $L \to \infty$ follows from Proposition \ref{prop:unstable_horoball_equid}.
Disintegrating the Masur-Veech measure fiberwise and using Proposition \ref{prop:ABEMb} we deduce
\begin{align*}
    \frac{1}{b_g} \int_{\QoM} f(\pi(q),&[\Im(q)])  \thinspace \lambda^+(q) \, d \widehat{\nu}_{\mathrm{MV}}(q) \\
    &= \frac{1}{b_g} \int_{\mathcal{M}_g} \int_{S(X)} f(X,[\Im(q)])  \thinspace \lambda^+(q) \thinspace ds_X(q) \, d \widehat{\mathbf{m}}(X)\\
    &= \frac{1}{b_g} \int_{\mathcal{P}^1\mathcal{M}_g} f(X,[\lambda]) \thinspace d \widehat{\mathbf{n}}(X,[\lambda]).
\end{align*}
Putting the identities above together finishes the proof.
\end{proof}

The following simultaneous equidistribution result can be proved by using similar arguments as in the proof of Theorem \ref{theo:main_2A} but working over the bundle $\mathcal{P}^1\mathcal{M}_g$ instead of $\mathcal{M}_g$ and using Proposition \ref{prop:XX} in place of Corollary \ref{cor:MV_horoball_equid}; compare to \cite[Proof of Theorem 3.5]{Ara20a}.

\begin{theorem}
\label{theo:main_3_new}
Let $\agamma:=(\agamma_1,\dots,\agamma_k)$ be an ordered, oriented simple closed multi-curve on $S_g$ and $X \in \mathcal{M}_g$ be a complex structure on $S_g$. Then, with respect to the weak-$\star$ topology for measures on $\mathcal{PMF}_g \times \mathcal{MRG}(S_g \setminus \agamma;\Delta)$,
\[
\lim_{L \to \infty} \frac{\xi_{\vg,X}^L}{s(X,\vg,L)} = \frac{\mu_\mathrm{Thu}^X}{\Lambda_g} \otimes \frac{\mathring{\eta}_{\mathrm{Kon}}^{\Delta}}{m_{\agamma}}.
\]
\end{theorem}

As a consequence, we see that even when prescribing how a set of curves coarsely wraps around $S_g$ (for example, by fixing a maximal train track chart for $\MF_g$), the critical graphs defined by those curves remain uniformly distributed.

\bibliographystyle{amsalpha}

\bibliography{bibliography}
\end{document}